\documentclass[10pt]{amsart}
\usepackage{amssymb,amsmath,amsopn,textcomp, thmtools}
\usepackage[foot]{amsaddr}
\usepackage{aliascnt}
\usepackage{mathrsfs} 
\theoremstyle{plain}
\newtheorem{theorem}{Theorem}[section]
\newtheorem{defi}[theorem]{Definition}

\newtheorem{lemma}[theorem]{Lemma}

\newtheorem{corollary}[theorem]{Corollary}
\newtheorem{proposition}[theorem]{Proposition}
\theoremstyle{definition}
\newtheorem{remark}[theorem]{Remark}
\newtheorem{example}[theorem]{Example}

\usepackage{bbm}
\newcommand{\R}{\mathbb{R}}
\newcommand{\N}{\mathbb{N}}

\usepackage{extarrows}
\usepackage{dsfont}
\usepackage{tikz,subcaption} 

\usepackage{mathabx} 

\allowdisplaybreaks

\begin{document}

\title{Entropy-information inequalities under  curvature-dimension conditions for continuous-time Markov chains}
\date{\today}
\author{Frederic Weber}
\email{frederic.weber@uni-ulm.de}
\thanks{The author is supported by a PhD-scholarship of the ``Studienstiftung des
deutschen Volkes'', Germany.}
\address[Frederic Weber]{Institut f\"ur Angewandte Analysis, Universit\"at Ulm, Helmholtzstra\ss{}e 18, 89081 Ulm, Germany.}


\begin{abstract}
In the setting of reversible continuous-time Markov chains, the $CD_\Upsilon$ condition has been shown recently to be a consistent analogue to the Bakry-\'Emery  condition in the diffusive setting in terms of proving Li-Yau inequalities under a finite dimension term and proving the modified logarithmic Sobolev inequality  under a positive curvature bound. In this article we examine the case where both is given, a finite dimension term and a positive curvature bound. For this purpose we introduce the $CD_\Upsilon(\kappa,F)$ condition, where the dimension term is expressed by a so called $CD$-function $F$. We derive functional inequalities relating the entropy to the Fisher information, which we will call entropy-information inequalities. Further, we deduce applications of entropy-information inequalities such as ultracontractivity bounds, exponential integrability of Lipschitz functions, finite diameter bounds and a modified version of the celebrated Nash inequality.
\end{abstract}

\maketitle

 

\section{Introduction}
\subsection{The curvature-dimension condition of Bakry-\'Emery}
The origins of the $\Gamma$-calculus of Bakry and \'Emery date back to the seminal work \cite{BaEm}. Meanwhile, this theory, for which the monograph \cite{BGL} is an excellent source, has been proven itself as a beautiful link between probability theory, geometry and analysis.

For motivational purposes we briefly survey the setting  of the Bakry-\'Emery theory in the sequel. Denoting by $L$ the infinitesimal generator of a Markov semigroup, the carr\'e-du-champ operator $\Gamma$ and the iterated carr\'e-du-champ operator $\Gamma_2$ are defined as
\begin{equation*}
\begin{split}
\Gamma(f,g)&= \frac{1}{2}\big( L(fg) -f \,Lg - g \,L f\big),\\
\Gamma_2(f,g)&= \frac{1}{2}\big( L \Gamma(f,g)-\Gamma(f,Lg)-\Gamma(g,Lf)\big)
\end{split}
\end{equation*}
for $f$ and $g$ lying in a suitable algebra of real-valued functions. One defines $\Gamma(f):=\Gamma(f,f)$ and $\Gamma_2(f):=\Gamma_2(f,f)$. Given a fixed invariant and reversible measure $\mu$ for the semigroup generated by $L$, the operator $L$ is said to satisfy the curvature-dimension condition $CD(\kappa,n)$ for $\kappa\in \R$ and $n\in [1,\infty]$, if 
\begin{equation}\label{eq:classicCDcondition}
\Gamma_2(f)\geq \kappa \Gamma(f) + \frac{1}{n}\big( L f\big)^2, \, \mu \mbox{-a.e.},
\end{equation}
holds in a sufficiently rich class of functions. 

The classical theory is based on the key assumption that the chain rule
\begin{equation}\label{eq:classicaldiffusionproperty}
L H(f) = H'(f) L f + H''(f) \Gamma(f)
\end{equation} 
is satisfied for every $H \in C^2(\R)$ and $f$ lying in a suitable class of functions.  Typical examples  for operators fitting to the abstract framework of \cite{BGL}, in particular satisfying \eqref{eq:classicaldiffusionproperty}, are given by second-order differential operators with smooth coefficients and without zero-th order term. For instance, consider the Laplace-Beltrami operator $L=\Delta_g$ on a Riemannian manifold $(M,g)$ with invariant and reversible measure given by the canonical Riemannian measure $\mu_g$. In this case one can show by means of the Bochner-Lichnerowicz formula that $CD(\kappa,n)$ is equivalent to $\mathrm{Ric}_g(x) \geq \kappa g(x)$ for almost every $x \in M$ and $\mathrm{dim}\, M \leq n$, where $\mathrm{Ric}_g$ denotes the Ricci-curvature tensor and $\mathrm{dim}\, M$ the topological dimension of the manifold $M$. In this sense, calling \eqref{eq:classicCDcondition} a curvature-dimension condition is well motivated.

The curvature-dimension inequality \eqref{eq:classicCDcondition} serves as a powerful tool to establish various functional inequalities. In fact, if $CD(\kappa,\infty)$ holds with $\kappa>0$, then the spectral gap inequality and the logarithmic Sobolev inequality are both satisfied with constant $\kappa$ (see \cite[Chapter 4 and 5]{BGL}). In the case that $CD(0,n)$ holds for $n<\infty$, one can deduce the Li-Yau inequality, which in turn leads to the parabolic Harnack inequality, cf. \cite[Chapter 6]{BGL}. Our main interest lies in the strongest case of having both, a positive curvature and a finite dimension term. Assuming that  $CD(\kappa,n)$ holds with $\kappa>0$ and $n<\infty$, one obtains  Sobolev inequalities or, equivalently, the logarithmic entropy-energy inequality, which reads as
\begin{equation}\label{eq:classicallogEE}
\mathrm{Ent}_\mu(f^2)\leq \frac{n}{2} \log\Big(1+ \frac{4}{\kappa n} \mathcal{E}(f)\Big),
\end{equation} 
for $f$ being in a sufficiently large class of functions with $\int_X f^2 \mathrm{d}\mu =1$ (where $\mathrm{Ent}_\mu$ denotes the Boltzmann entropy and $\mathcal{E}$ the Dirichlet energy on $L^2(\mu)$ associated with $L$ and $\mu$). The functional inequality \eqref{eq:classicallogEE} is an important instance of what is called  an entropy-energy inequality, that is
\begin{equation}\label{eq:EntropieEnergie}
\mathrm{Ent}_\mu(f^2) \leq \Phi \big( \mathcal{E}(f)\big),
\end{equation}
where $\Phi:(0,\infty)\to \R$ is a strictly increasing and concave $C^1$-function. We refer to \cite[Chapter 7]{BGL}, where applications of entropy-energy inequalities, such as ultracontractivity or diameter bounds, have been discussed.

\subsection{Existing approaches for finding substitutes for curvature and dimension in the discrete setting}
The issue of finding suitable substitutes of Ricci curvature lower bounds in the discrete setting has been a very vibrant topic of research in the last decade and a half, see e.g. the recent book \cite{ModCur}.
 
Based on the powerful approach of optimal transport, for which we refer to the seminal works \cite{LoVi,Stu1,Stu2,Vil}, Erbar and Maas successfully developed the theory of entropic Ricci curvature in the context of finite Markov chains in \cite{ErMa} and \cite{Maa}. Another important notion of discrete curvature that relies on ideas from optimal transport  is due to Ollivier (see \cite{Oll}). The latter curvature notion has been studied in a variety of articles concerning the case that the underlying state space is given by  a locally finite graph, see e.g. \cite{JoLi,MuWo}. 

With regard to the Bakry-\'Emery approach, it is apparently still possible to define the operators $\Gamma$ and $\Gamma_2$ in the discrete setting, where $L$ now denotes the generator of a Markov chain. However, even though some positive results such as eigenvalue estimates in \cite{LiYa} or diameter bounds in \cite{LMP} can be deduced, the Bakry-\'Emery condition is not as applicable as in the continuous setting, in particular with regard to Li-Yau inequalities and (modified) logarithmic Sobolev inequalities. This is caused especially by  the major difficulty that the diffusion property \eqref{eq:classicaldiffusionproperty} does not hold in the discrete setting. There are several modified versions of curvature-dimension inequalities in the discrete setting which are based on the approach of identifying certain discrete substitutes for the chain rule, e.g. in the context of Li-Yau inequalities we refer to \cite{Harvard,DKZ,Mn1}. In particular, in \cite{DKZ} the identity
\begin{equation}\label{eq:FFI}
L(\log f)=\frac{L f}{f} - \Psi_\Upsilon(\log f)
\end{equation}
has been used as the appropriate replacement for the case of $H=\log$ in \eqref{eq:classicaldiffusionproperty}. Here the operator $\Psi_\Upsilon$ is defined as in \eqref{eq:PsiH} below, with $H=\Upsilon$, where  $\Upsilon:\R\to \R$ is given by $\Upsilon(r)= e^r-r-1$, $r\in \R$. We will comment on regularity assumptions for $f$ ensuring the validity of \eqref{eq:FFI} in the next subsection. One of the key ideas of \cite{DKZ} is to make use of so called $CD$-functions in order to express the dimension term in their CD condition. This in fact leads to significantly improved estimates regarding the corresponding Li-Yau inequalities, which are even sharp in some instances. We will follow the approach of using $CD$-functions in this article as well.

Regarding positive curvature bounds, based on the identity \eqref{eq:FFI}, it has been shown very recently in \cite{WZ} that the $CD_\Upsilon(\kappa,\infty)$ condition serves as a consistent analogue to the classical curvature-dimension condition with regard to the strategy of proving entropy decay of an exponential rate using the entropy method. The resulting functional inequality, the modified logarithmic Sobolev inequality, holds with constant $\kappa>0$ provided that $CD_\Upsilon(\kappa,\infty)$ is satisfied (see \cite[Corollary 3.5]{WZ}). In the sense of the relation between curvature-dimension inequalities and related functional inequalities, the modified logarithmic Sobolev inequality with regard to the $CD_\Upsilon(\kappa,\infty)$ condition in the discrete setting serves as the appropriate counterpart to the logarithmic Sobolev inequality with regard to the $CD(\kappa,\infty)$ condition in the diffusive setting.

Moreover, we refer to the discussion in \cite[Remark 2.9]{WZ} that shows that the $CD_\Upsilon$ condition with finite dimension term  is strongly related to the articles \cite{Mn1} and \cite{DKZ} (see also Remark \ref{CDdiscussionremark} below) and in particular implies Li-Yau type inequalities. In this sense, the $CD_\Upsilon$ condition serves as a suitable analogue to the Bakry-\'Emery condition with regard to both, positive curvature in terms of the entropy method and finite dimension in terms of Li-Yau inequalities. The main motivation of this paper is to  identify the appropriate discrete counterpart to the logarithmic entropy-energy inequality \eqref{eq:classicallogEE} with regard to the $CD(\kappa,n)$ condition in the diffusive setting, or in other words, to investigate the case where both is satisfied, a positive curvature bound and a finite dimension term.
\subsection{Setting and main results}
We consider a time-homogeneous, continuous-time Markov chain $\big( Z_t\big)_{t \geq 0}$ defined on a probability space $\big(\Omega, \mathcal{F},\mathbb{P}\big)$ and  with  (finite or infinite) countable state space $X$. The generator $L$ of the Markov chain  is defined on a suitable class of functions $f:X \to \R$  by
\begin{equation}\label{eq:generator}
L f(x) = \sum_{y \in X} k(x,y)f(y)= \sum_{y \in X} k(x,y) \big( f(y)-f(x)\big).
\end{equation}
Here, we assume $\sum_{y \in X}k(x,y) = 0$, where $k(x,y)\geq 0$ denotes the transition rate for jumping from $x$ to $y$ if $x\neq y$. We remark that $L$ determines naturally a graph structure with vertex set $X$ and edge weights  given by $k(x,y)$ for $x,y \in X$, $x\neq y$, to which we will refer as the underlying graph to $L$. If $k(x,y) \in \{0,1\}$ for any $x,y \in X$ with $x \neq y$, then the underlying graph to $L$ is an unweighted graph.

 We denote the associated (sub-)Markov semigroup on the space of bounded functions by $\big(P_t\big)_{t\geq 0}$, which is given by
\begin{equation}\label{SubMarkov}
P_t f(x) = \mathbb{E}(f(Z_t)|Z_0=x\big).
\end{equation}

Further, we suppose that the Markov chain is irreducible and that a unique invariant measure $\mu$ exists such that the detailed balance condition
\begin{equation}\label{eq:dbcondition}
\mu(\{x\})k(x,y)=\mu(\{y\})k(y,x)
\end{equation}
is valid for any $x,y \in X$. Let $\pi:X\to (0,\infty)$ denote the density for $\mu$ with respect to the counting measure on $X$, i.e. $\mathrm{d}\mu = \pi \mathrm{d}\#$. It is a basic consequence of irreducibility and reversibility that $\pi(x)>0$ for any $x \in X$.

It is well known that the Markov chain is  positive recurrent if and only if $\mu$ is finite (and hence can be assumed to be a probability measure) and the Markov chain is non-explosive (see e.g. \cite{Nor}). In particular, provided that the Markov chain is positive recurrent, stochastic completeness (that is $P_t \mathds{1}=\mathds{1}$) and ergodicity (by which we mean what is sometimes called ordinary ergodicity, see e.g. \cite{And}) hold true. In the recent work \cite{WZ}, which is strongly related to this article, positive recurrence has been an important assumption. If one allows for the Markov chain being explosive, then the semigroup given by \eqref{SubMarkov} is only submarkovian, which ensures still that $P_t f$ is bounded provided that $f$ is bounded. For more details on the general theory of continuous-time Markov chains we refer the reader to \cite{And} and \cite{Nor}. 

We denote by $\R^X$ the space of real-valued functions on $X$ and by $\ell^p(\mu)$, $1\leq p < \infty$, the elements of $\R^X$ that are $p$-summable with respect to $\mu$. Further, $\ell^\infty(X)$ denotes the space of bounded real-valued functions on $X$. Throughout this article the mapping $\Vert \cdot \Vert_p : \ell^p(\mu)\to [0,\infty)$, $1\leq p< \infty$, denotes the $\ell^p(\mu)$-norm.

Moreover, we denote by $\mathcal{P}(X)$ the set of probability densities with respect to $\mu$, by $\mathcal{P}_*(X)$ the set of  elements in $\mathcal{P}(X)$ that are strictly positive at any $x \in X$ and $P_*^+(X):= \mathcal{P}_*(X) \cap \ell^{\infty,+}(X)$, where
\begin{equation*}
\ell^{\infty,+}(X)=\{f \in \ell^{\infty}(X): \exists c>0 \text{ such that } f(x) \geq c, \forall x \in X\}.
\end{equation*}

We assume throughout this paper that at any $x \in X$
\begin{equation}\label{eq:M1assumption}
M_1(x):=\sum_{y \in X \setminus \{x\}} k(x,y) <\infty
\end{equation}
and
\begin{equation}\label{eq:M2assumption}
M_2(x):= \sum_{y \in X \setminus \{x\}}k(x,y)\sum_{z \in X \setminus\{y\}}k(y,z) < \infty.
\end{equation}
Further, we define 
$
M_{1,\inf}:= \inf_{x \in X} M_1(x)\in [0,\infty)
$
and 
$
M_{1,\sup}:= \sup_{x \in X} M_1(x) \in (0,\infty] 
$ and introduce
\begin{equation*}
N(x):= \sum_{y \in X \setminus \{x \}} k(x,y) k(y,x) \leq M_2(x)<\infty.
\end{equation*}

We recall from \cite{WZ} the definition of the operators $\Psi_H$ and $\Psi_{2,H}$, where $H:\R \to \R$ is a continuous mapping. We define
\begin{equation}\label{eq:PsiH}
\Psi_H(f)(x) = \sum_{y \in X} k(x,y) H(f(y)-f(x)),\, x \in X,
\end{equation}
for any $f \in \ell^\infty(X)$ and  
\begin{equation}\label{eq:BHDefinition}
B_{H}(f,g)(x)= \sum_{y \in X} k(x,y) H(f(y)-f(x))(g(y)-g(x)), \, x \in X,
\end{equation}
for suitable functions $f$ and $g$. In particular, the conditions \eqref{eq:M1assumption} and \eqref{eq:M2assumption} ensure that we can choose in \eqref{eq:BHDefinition} $g=L f$ for $f \in \ell^\infty(X)$. This guarantees that for $f\in \ell^\infty(X)$ and $x \in X$ the operator 
\begin{equation*}
\Psi_{2,H}(f)(x)= \frac{1}{2}\big( L \Psi_H(f)(x) - B_{H'}(f,Lf)(x)\big)
\end{equation*}
is well defined. In the case of $H(r)=\frac{1}{2}r^2$, $\Psi_H(f)$ coincides with $\Gamma(f)$ and $\Psi_{2,H}(f)$ with $\Gamma_2(f)$. For our purposes, the mapping $\Upsilon(r)=e^r-r-1$, $r\in \R$, will play a key role. Indeed, the choice of $H(r)=\Upsilon(r)$, which is motivated by the identity \eqref{eq:FFI}, leads to the operators $\Psi_\Upsilon(f)$ and $\Psi_{2,\Upsilon}(f)$, which are the central objects of investigation in the recent article \cite{WZ}. Let us remark that in our setting the identity \eqref{eq:FFI} holds true for any $f \in \R^X$ such that $f,\log f \in \ell^1(k(x,\cdot))$ for any $x \in X$ (cf. \cite[Lemma 2.2]{WZ}), which is for instance the case when $f\in \ell^{\infty,+}(X)$.

We recall a representation formula for the operator $\Psi_{2,\Upsilon}$, which has been used frequently in order to study a large class of examples in \cite{WZ}, and reads as
\begin{equation}\label{eq:Psi2Formel}
\begin{split}
2\Psi_{2,\Upsilon}(f)(x)&=\sum_{y\in X\setminus\{x\}} k(x,y) \sum_{z\in X} k(y,z)\Big(\Upsilon\big(f(z)-f(y)\big)-
\Upsilon'\big(f(y)-f(x)\big)\big(f(z)-f(y)\big)\Big)\\
& \quad\; +\sum_{y\in X\setminus\{x\}} k(x,y) \Upsilon'\big(f(y)-f(x)\big)\sum_{z\in X} k(x,z)\big(f(z)-f(x)\big)\\
& \quad\; - \sum_{y\in X\setminus\{x\}} k(x,y) \sum_{z\in X} k(x,z)\Upsilon\big(f(z)-f(x)\big).
\end{split}
\end{equation}

The detailed balance condition \eqref{eq:dbcondition} ensures that the generator of the Dirichlet form given by
\begin{equation}\label{Dirichletform}
\mathcal{E}(f,g) = \frac{1}{2}\sum_{x \in X}\sum_{y \in X} k(x,y) \big(f(y)-f(x)\big)\big(g(y)-g(x)\big) \pi(x)
\end{equation}
for $f,g$ being suitable functions, coincides with $L$ given by \eqref{eq:generator} on bounded functions which are contained in the domain of the form generator, see \cite{KeLe}. We will also denote the $\ell^2(\mu)$ operator by $L$ in the sequel. Further, as the corresponding $\ell^2(\mu)$-semigroup generated by $L$ is an extension of the semigroup given by \eqref{SubMarkov} restricted to $\ell^\infty(X) \cap \ell^2(\mu)$, we will also use the notation $(P_t)_{t \geq 0}$ for the corresponding $\ell^2(\mu)$-semigroup. 

An eminent role will be played by the entropy, being given as
\begin{equation*}
\mathrm{Ent}_\mu(f)= \int_X f \log f \mathrm{d}\mu - \int_X f \mathrm{d}\mu \,\log \int_X f \mathrm{d}\mu,
\end{equation*}
and the Fisher information $\mathcal{I}(f)=\mathcal{E}(f,\log f)$.  We refer to the beginning of Section \ref{sec:EIinequalities} for more details on some elementary properties of the entropy resp. the Fisher information and on corresponding admissible functions.

We say that $L$ satisfies $CD_\Upsilon(\kappa,F)$ if
\begin{equation*}
\Psi_{2,\Upsilon}(f) \geq \kappa \Psi_\Upsilon(f) + F_0\big( - L f\big)
\end{equation*}
holds on $X$ for any $f \in \ell^\infty(X)$, where $\kappa \in \R$ and $F_0:\R\to[0,\infty)$ is the trivial extension of a $CD$-function $F:[0,\infty) \to [0,\infty)$ (see Definition \ref{def:CDfunction} below), i.e. $F_0(r)=0$ if $r<0$. Note that the notation of the condition $CD_\Upsilon(\kappa,\infty)$, which states that $\Psi_{2,\Upsilon}(f)\geq \kappa \Psi_\Upsilon(f)$ holds on $X$ for any $f\in \ell^\infty(X)$, is a bit missleading since it really means that the dimension term vanishes. This terminology is clearly motivated from the case  of the quadratic $CD$-functions $F(r)=\frac{r^2}{n}$, $n \in [1,\infty)$, to which we will refer as the $CD_\Upsilon(\kappa,n)$ condition (motivated by the classical Bakry-\'Emery notation). Note that $CD_\Upsilon(\kappa,n)$ has already been mentioned  in slightly different form in \cite[Remark 2.9]{WZ} and also in \cite{Mn1} in a rather implicit form (cf. \cite[Section 9]{WZ}). 

We now describe our main results. Assuming that the $CD$-function is convex, continuously differentiable and such that the mapping $r\mapsto \frac{F(r)}{r^{1+\delta}}$  is increasing on $(0,\infty)$ for some $\delta>0$, we will be able to show in Theorem \ref{maininequalityunderCD} that $CD_\Upsilon(\kappa,F)$  (with $\kappa>0$) implies the bound
\begin{equation}\label{eq:introductmaininequality}
\mathrm{Ent}_\mu(f) \leq \int_0^\infty G \Big( \frac{\kappa}{e^{2\delta\kappa t} \big( 1+ \frac{\kappa \mathcal{I}(f)}{F(\mathcal{I}(f))}\big) - 1}\Big)\mathrm{d}t
\end{equation}
for any $f \in P_*(X)$ with $\mathrm{Ent}_\mu(f)<\infty$ and $\mathcal{I}(f)\in(0,\infty)$, where $G:(0,\infty)\to (0,\infty)$ denotes the inverse function of the mapping $r \mapsto \frac{F(r)}{r}$, $r>0$. In particular, in the case of $CD_\Upsilon(\kappa,n)$ with $\kappa>0$ and $n<\infty$, \eqref{eq:introductmaininequality} reads as
\begin{equation}\label{eq:intruductloginequality}
\mathrm{Ent}_\mu(f) \leq \frac{n}{2}\log \Big( 1+ \frac{\mathcal{I}(f)}{\kappa n}\Big),
\end{equation}
see Corollary \ref{logEIunderCD(kappa,n)}. Consequently, \eqref{eq:intruductloginequality} with regard to $CD_\Upsilon(\kappa,n)$ is the natural discrete analogue to \eqref{eq:classicallogEE} with regard to $CD(\kappa,n)$ in the diffusive setting. Note that in the diffusive setting the chain rule implies $\mathcal{I}(f^2)=4\mathcal{E}(f)$ for suitable functions, which yields that in the classical situation the inequalities \eqref{eq:classicallogEE} and \eqref{eq:intruductloginequality} coincide. In particular, \eqref{eq:intruductloginequality} is an important example of what will be called an entropy-information inequality, i.e. a functional inequality of the form $\mathrm{Ent}_\mu(f)\leq \Phi(\mathcal{I}(f))$ for a strictly increasing and concave $C^1\big((0,\infty)\big)$-function $\Phi$, to which we will refer as the growth function.

As the modified logarithmic Sobolev inequality differs from the logarithmic Sobolev inequality in the discrete setting, hypercontractivity of the semigroup, which is equivalent to the latter also in the discrete setting (see \cite{DSC96}), does not characterize the modified logarithmic Sobolev inequality. In \cite{BoTe}, Bobkov and Tetali established a hypercontractivity formulation for $e^{P_t f}$ being suitable for the modified logarithmic Sobolev inequality. In this spirit, we show in Theorem \ref{ultracontractivityresult} that certain ultracontractivity bounds for $e^{P_t f}$ hold under corresponding entropy-information inequalities. In particular, in case of the $CD_\Upsilon(\kappa,n)$ condition with $\kappa>0$ and $n<\infty$, we will be able to derive that
\begin{equation*}
\Vert e^{P_t f} \Vert_\infty \leq \Big( 1+ \frac{1}{2\kappa t}\Big)^\frac{n}{2} \Vert e^f \Vert_1 
\end{equation*}
holds for any $f\in \ell^\infty(X)$ and any $t>0$. 

In Theorem \ref{expintegrabilityforLipschitz} we prove that $\mu$-integrable $1$-Lipschitz functions, by which we mean that $\Vert \Gamma(f)\Vert_\infty \leq 1$, are exponentially integrable. Moreover, provided that the growth function satisfies $\int_0^\infty \frac{\Phi(s^2)}{s^2}\mathrm{d}s<\infty$, we show
that
\begin{equation*}
\Vert f - \int_X f \mathrm{d}\mu \Vert_\infty \leq \int_0^\infty \frac{\Phi(s^2)}{s^2}\mathrm{d}s
\end{equation*}
holds for any $1$-Lipschitz function $f$. This in turn implies
a finite diameter bound, which  in the special case of $CD_\Upsilon(\kappa,n)$ with $\kappa>0$ and $n<\infty$ reads as
\begin{equation}\label{eq:introductCDdiametrbound}
\mathrm{diam}_\varrho \leq \pi \sqrt{\frac{n}{\kappa}},
\end{equation}
see Corollary \ref{Diameterbound}. Interestingly, \eqref{eq:introductCDdiametrbound} coincides  with the diameter bound that has been obtained by Liu, M\"unch and Peyerimhoff in \cite{LMP} for the weaker $CD(\kappa,n)$ condition but under assumptions on the underlying graph to $L$ which can be expected to be more restrictive, see Remark \ref{rem:LMPcomparision}.

Finally, we also show that a new modified version of the celebrated Nash inequality holds under an entropy-information inequality with logarithmic growth function. In the particular case of the $CD_\Upsilon(\kappa,n)$ condition with $\kappa>0$ and $n<\infty$, this modified Nash inequality says that
\begin{equation*}
\Vert f \Vert_2^{n+2} \leq \Big( \Vert f \Vert_2^2 + \frac{\mathcal{I}(f^2)}{\kappa n}\Big)^{\frac{n}{2}} \Vert f \Vert_1^2
\end{equation*}
holds for any non-vanishing $f \in \ell^2(\mu)$, see Theorem \ref{modNash} and Corollary \ref{modNashunderCD}.

The article is organized as follows. In Section \ref{sec:CDcondition} we introduce the curvature-dimension condition $CD_\Upsilon(\kappa,F)$ and discuss several examples. Next, we define the notion of entropy-information inequalities and investigate their relation to the $CD_\Upsilon(\kappa,F)$ condition in the case of power-type $CD$-functions. In the remaining part of the paper, we discuss applications of entropy-information inequalities and hence also of the $CD_\Upsilon(\kappa,F)$ condition. We derive ultracontractive bounds for $e^{P_t f}$ in Section \ref{sec:ultracontractive}, exponential integrability of Lipschitz functions and diameter bounds in  Section \ref{sec:diameter} and, finally, a modified version of the  Nash inequality in Section \ref{sec:Nash}.

\section{The $CD_\Upsilon$ condition with finite dimension and some examples}\label{sec:CDcondition}
In this section we generalize the $CD_\Upsilon(\kappa,\infty)$ condition from the very recent work \cite{WZ} by adding a general dimension term involving a $CD$-function. We first recall the notion of a $CD$-function that originates from the work of \cite{DKZ} and also has been mentioned in \cite[Remark 2.9]{WZ}.
\begin{defi}\label{def:CDfunction}
A continuous function $F:[0,\infty) \to [0,\infty)$ is called $CD$-function, if $F(0)=0$, $r\mapsto \frac{F(r)}{r}$ is strictly increasing on $(0,\infty)$ and $\frac{1}{F}$ is integrable at $\infty$. 
\end{defi}
For a given $CD$-function $F$, we will call the function $F_0:\R\to [0,\infty)$ given by $F_0(r)=F(r)$ if $r\geq 0$ and $F_0(r)=0$ otherwise the trivial extension of $F$.
\begin{remark}\label{remarkCDfunctions}
If a function $F:[0,\infty)\to[0,\infty)$ with $F(0)=0$ is strictly convex on $(0,\infty)$, the mapping $r\mapsto \frac{F(r)}{r}$ is strictly increasing on $(0,\infty)$, cf. \cite[Remark 3.3]{DKZ}. However, it can not be deduced in general that $F$ is a $CD$-function as, for instance, the function $r \mapsto \Upsilon(-r)$, $r \in [0,\infty)$, serves as a counterexample. 

If we impose instead that $F(0)=0$ and $r \mapsto \frac{F(r)}{r^{1+\delta}}$ is increasing on $(0,\infty)$ for some $\delta>0$, as it will be done in Theorem \ref{maininequalityunderCD} (cf. Remark \ref{rem:CDfunctionassumption}), then it follows that $F$ is a $CD$-function. Indeed, it is obvious that $r\mapsto \frac{F(r)}{r}$ is strictly increasing on $(0,\infty)$ and further we have
\begin{equation*}
\int_c^\infty \frac{1}{F(r)}\mathrm{d}r \leq \frac{c^{1+\delta}}{F(c)}\int_c^\infty r^{-(1+\delta)}\mathrm{d}r <\infty,
\end{equation*}
where $c>0$. 
\end{remark}
\begin{defi}\label{definitionCDcondition}
We say that the Markov generator $L$ satisfies $CD_\Upsilon(\kappa,F)$ at $x\in X$ for $\kappa \in \R$ and a $CD$-function $F: [0,\infty) \to [0,\infty)$ with trivial extension $F_0$ if
\begin{equation}\label{eq:definitionCD}
\Psi_{2,\Upsilon}(f)(x) \geq \kappa \Psi_\Upsilon(f)(x) + F_0 \big( - L f (x)\big)
\end{equation}
holds for all $f\in \ell^\infty(X)$. If $L$ satisfies $CD_\Upsilon(\kappa,F)$ at any $x \in X$, then we say that $L$ satisfies $CD_\Upsilon(\kappa,F)$.\\
In the special case of $F(r)=\frac{1}{n}r^2$, $r \geq 0$, for some $n \in [1,\infty)$, we say that $L$ satisfies $CD_\Upsilon(\kappa,n)$.
\end{defi}
According to the ambiguity in the notation, we emphasize that throughout this article a capital letter $F$ in the condition $CD_\Upsilon(\kappa,F)$  always refers to a $CD$-function, while a small letter $n$ in the condition $CD_\Upsilon(\kappa,n)$ refers to the constant of a quadratic $CD$-function.

\begin{remark}\label{CDdiscussionremark}(i) It must be pointed out that the condition $CD_\Upsilon(\kappa,n)$ has been introduced in \cite[Remark 2.9]{WZ} in seemingly stricter form, as we only require the dimension term in Definition \ref{definitionCDcondition} for functions $f \in \ell^\infty(X)$ with $-L f(x)>0$. Further, the latter condition is the only difference of $CD_\Upsilon(\kappa,n)$ and the condition $CD\psi(n,\kappa)$ for the specific choice of $\psi=\log$, which has been introduced in the case of finite and unweighted graphs in \cite{Mn1}. We refer to \cite[Section 9]{WZ} for a detailed account on the relation of the operators appearing in \eqref{eq:definitionCD} and those in \cite{Mn1}. The condition $-L f(x)>0$ also appears in the formulation of other curvature-dimension conditions, such as for instance in the case of the exponential curvature-dimension condition of \cite{Harvard} and the $CD(F;0)$ condition in \cite{DKZ}, where $F$ denotes a $CD$-function. We refer to \cite[Remark 2.9(iii)]{WZ}, which shows that the condition $CD_\Upsilon(0,F)$ suffices to deduce Li-Yau inequalities using the results of \cite{DKZ}. 

(ii) Importantly, the $CD_\Upsilon(\kappa,n)$ condition (or more generally $CD_\Upsilon(\kappa,F)$ with $F(r) \sim \frac{1}{n}r^2$ as $r \to 0^+$) implies the Bakry-\'Emery condition $CD(\kappa,n)$. This fact relies on the identities  
\begin{equation}\label{eq:scalingofPsis}
\lim\limits_{\lambda \to 0}\frac{\Psi_{2,\Upsilon}(\lambda f)}{\lambda^2}=\Gamma_2(f),\, \lim\limits_{\lambda \to 0}\frac{\Psi_{\Upsilon}(\lambda f)}{\lambda^2}= \Gamma(f),
\end{equation}
holding true for any $f \in \ell^\infty(X)$, which has been shown in the proof of \cite[Proposition 2.11]{WZ}. As it has been pointed out in \cite[Remark 2.13]{WZ}, the procedure extends easily to a quadratic dimension term. More accurately, in the formulation of Definition \ref{definitionCDcondition}, it first implies $CD(\kappa,n)$ only for $f\in \ell^\infty(X)$ with $-L f (x)>0$, but then extends to any $f\in \ell^\infty(X)$ by linearity of $L$, bilinearity of $\Gamma$ and the definition of $\Gamma_2$. The fact that the Bakry-\'Emery condition is necessary for $CD_\Upsilon$ with a quadratic $CD$-function also motivates to study the former condition. In particular, we refer to \cite{SWZ1}, where $CD(0,n)$ has been studied for a large class of operators with long range jumps and state space $\mathbb{Z}$.

(iii) The property \eqref{eq:scalingofPsis} has an important consequence for $CD$-functions that behave like a power-type function near the origin. Indeed, if  $CD_\Upsilon(\kappa,F)$ holds with $\kappa \in \R$ and $F(r)\sim r^{1+\delta}$ as $r\to 0^+$ for some $\delta>0$, then we infer from \eqref{eq:scalingofPsis} that for any $f\in \ell^\infty(X)$ with $-L f(x)>0$, $x \in X$, we have at $x$
\begin{align*}
\Gamma_2(f)&=\lim\limits_{\lambda \to 0^+}\frac{\Psi_{2,\Upsilon}(\lambda f)}{\lambda^2}\\ 
&\geq \lim\limits_{\lambda \to 0^+} \frac{\kappa \Psi_\Upsilon(\lambda f) + F(-\lambda L f)}{\lambda^2}\\
&= \kappa \Gamma(f) + (-L f)^{1+\delta}\lim\limits_{\lambda \to 0^+} \frac{F(-\lambda L f)}{(-\lambda L f)^{1+\delta}} \lambda^{1+\delta-2}.
\end{align*}
Consequently, $\delta \geq 1$ must hold, or in other words the best behavior of a $CD$-function near the origin one can hope for is quadratic.
\end{remark}
Several concrete examples have been considered in \cite[Section 5]{WZ} to study the $CD_\Upsilon(\kappa,\infty)$ condition. It has turned out that the functions
\begin{equation}\label{eq:nucdfunctions}
\nu_{c,d}(r)=c\Upsilon'(r)r + \Upsilon(-r)-d \Upsilon(r), \, r \in \R,
\end{equation} 
with constants $c,d \in \R$, are of eminent importance. We refer the reader to the Appendix of \cite{WZ}, where basic properties of these functions have been collected.

As a warm-up, we begin with the basic example of the two-point space, for which we provide another property of the functions given by \eqref{eq:nucdfunctions} in the Appendix below.
\begin{example}\label{ex:2point}
We consider the two-point space $X=\{0,1\}$ with $k(0,1)=a$ and $k(1,0)=b$, where $a,b>0$. 
Here, the invariant and reversible probability measure $\mu$ is given by $\mathrm{d}\mu=\pi \mathrm{d}\#$ with $\pi(0)=\frac{b}{a+b}$ and $\pi(1)=\frac{a}{a+b}$.
We write $\tilde{x}=1-x$ for $x\in X$ and $t=f(\tilde{x})-f(x)$. From \cite{WZ} we know that $CD_\Upsilon(\kappa,\infty)$ holds true for some $\kappa>0$. In order to show that the $CD_\Upsilon$ condition is fulfilled with a positive curvature constant and a finite dimension term it hence suffices to show $CD_\Upsilon(0,F)$, where $F$ is a $CD$-function. Thus it remains to compare $\Psi_{2,\Upsilon}(f)(x)$ with $F(-L f(x)) = F(-k(x,\tilde{x})t)$, for $F$ being specified below. We have
\begin{align*}
2\Psi_{2,\Upsilon}&(f)(x) = L\Psi_\Upsilon(f)(x)-B_{\Upsilon'}(f,Lf)(x)\\
& = k(x,\tilde{x})\Big(\big(\Psi_\Upsilon(f)(\tilde{x})-\Psi_\Upsilon(f)(x)  \big)
- \Upsilon'\big(f(\tilde{x})-f(x)  \big)\big(Lf(\tilde{x})-Lf(x)\big)\Big)\\
& = k(x,\tilde{x}) \,k(\tilde{x},x)\big(\Upsilon(-t)+\Upsilon'(t)t\big)+k(x,\tilde{x})^2\big(\Upsilon'(t)t-\Upsilon(t) \big)\\
& = k(x,\tilde{x})k(\tilde{x},x) \nu_{1+\frac{k(x,\tilde{x})}{k(\tilde{x},x)},\frac{k(x,\tilde{x})}{k(\tilde{x},x)}}(t).
\end{align*}
Note that for $c,d\in \R$ and $\eta>0$ we have that $\nu_{c+\eta,d+\eta}(r)\geq \nu_{c,d}(r)$ at any $r\in \R$, since $\Upsilon'(r)r \geq \Upsilon(r)$, $r \in \R$, holds by convexity.
Hence, we can estimate
\begin{equation*}
\Psi_{2,\Upsilon}(f)(x) \geq \frac{ab}{2}\nu_{1+\lambda,\lambda}(t),
\end{equation*}
where $\lambda:= \min\{\frac{a}{b},\frac{b}{a}\}$. Note that $\nu_{1+\lambda,\lambda}$ is strictly convex by Lemma \ref{nulambdaconvex}. Then, we infer from Remark \ref{remarkCDfunctions} and the asymptotic behavior of $\nu_{1+\lambda,\lambda}$ that the mapping $F:[0,\infty)\to[0,\infty)$ defined as
\begin{equation*}
F(r)= \frac{ab}{2}\,\nu_{1+\lambda,\lambda}\big(-\frac{r}{\max\{a,b\}}\big), \, r \geq 0,
\end{equation*}
is a $CD$-function and we deduce that $CD_\Upsilon(0,F)$ holds true.
\end{example}
We continue with the following basic observation.
\begin{proposition}\label{gammacriterion}
Let $x \in X$, $\kappa \in \R$, $\gamma: \R \to [0,\infty)$ be convex on $\R$ with $\left.\gamma\right|_{[0,\infty)}$ being a $CD$-function and $\alpha:X \to (0,\infty)$ such that 
\begin{equation}\label{eq:sufficientfordimensionalCD}
\Psi_{2,\Upsilon}(f)(x) \geq \kappa \Psi_\Upsilon(f)(x) + \alpha(x) \sum_{y \in X} k(x,y) \gamma(f(x)-f(y))
\end{equation}
holds for any $f \in \ell^\infty(X)$. Then $L$ satisfies $CD_\Upsilon(\kappa,F)$ at $x$ with $F(r)= \alpha(x) M_1(x) \gamma\big(\frac{r}{M_1(x)}\big)$, $r\geq 0$. If $0<M_{1,\inf}\leq M_{1,\sup}<\infty$, $\alpha_*:=\inf_{x \in X} \alpha(x)>0$ and  \eqref{eq:sufficientfordimensionalCD} holds for any $f \in \ell^\infty(X)$ and all $x \in X$, then $L$ satisfies $CD_\Upsilon(\kappa,F)$ with $F(r)= \alpha_* M_{1,\inf}\gamma\big(\frac{r}{M_{1,\sup}}\big)$, $r\geq 0$.
\end{proposition}
\begin{proof}
For fixed $x \in X$, we observe by Jensen's inequality
\begin{equation*}
\sum_{y \in X \setminus\{x\}}k(x,y) \gamma(f(x)-f(y)) \geq M_1(x) \gamma\Big(-\frac{L f(x)}{M_1(x)}\Big),
\end{equation*}
from which the first claim follows by \eqref{eq:sufficientfordimensionalCD}. Clearly, $\gamma$ is increasing on $[0,\infty)$, which implies the second claim. 
\end{proof}
Note that the tempting naive approach to deduce a $CD_\Upsilon$ condition with non-negative curvature bound and finite dimension term from $CD_\Upsilon(\kappa,\infty)$ (with $\kappa>0$) alone by using Proposition \ref{gammacriterion} does not work. Indeed, the mapping $r\mapsto \Upsilon(-r)$ can not play the role of the function $\gamma$ from Proposition \ref{gammacriterion} since it only grows linearly at $\infty$ and is hence not a $CD$-function. This is a difference to the Bakry-\'Emery condition in our setting, where the analogous result to Proposition \ref{gammacriterion} yields at least that $CD(\kappa,\infty)$(with $\kappa>0$) at $x\in X$ implies $CD(\lambda \kappa,(1-\lambda)n)$ at $x\in X$, where  $\lambda \in [0,1]$ and $n<\infty$ depends on $\kappa$ and $x$. For our purposes however, we need a refined analysis.

In light of the subsequent sections, power-type $CD$-functions are of particular importance. We note that for any $\delta\geq 1$ there exists some optimal $c_\delta>0$ such that the estimate
\begin{equation}\label{eq:Palmesqaureestimate}
\Upsilon(r)+\Upsilon(-r)\geq c_\delta |r|^{1+\delta}
\end{equation}
holds true for any $r \in \R$, which follows from the asymptotic and monotonic  behavior of the mapping $r \mapsto \Upsilon(r)+\Upsilon(-r)$, $r\in \R$, and the fact that $\Upsilon(r)+\Upsilon(-r)\sim r^2$ as $r \to 0$. For instance, it can be easily checked that the optimal constant for $\delta=1$ in \eqref{eq:Palmesqaureestimate} is given by $c_1=1$.

We illustrate the practical use of \eqref{eq:Palmesqaureestimate} in the following example.
\begin{example}\label{ex:keql}
Let $X$ be an arbitrary countable set with at least two elements and $l:X\to (0,\infty)$ being integrable on $X$ with respect to the counting measure. Further, we set $k(x,y)=l(y)$ for all $x,y \in X$, $x\neq y$. Then $\mu$ given by $\mathrm{d}\mu=\pi\mathrm{d}\#$ with $\pi(x)=l(x)$, $x \in X$, is an invariant and reversible measure. In \cite[Example 5.2]{WZ} it has been shown that $L$ satisfies $CD_\Upsilon(\sqrt{2|l|_1l_*},\infty)$, where $l_*=\inf_{x \in X} l(x)$ and $|l|_1$ denotes the $\ell^1$-norm with respect to the counting measure on $X$. Clearly, the integrability of $l$ implies that $l_*=0$ if  $X$ is infinite. It has also been shown that $CD_\Upsilon(0,\infty)$ is best possible (concerning the curvature term) in the infinite state space case. Therefore, we will only consider the case of $X$ being finite in the sequel, i.e.  $l_*>0$ holds.

In \cite[Example 5.2]{WZ}, the following representation formula has been established for $f \in \R^X$ at $x\in X$ from \eqref{eq:Psi2Formel}
\begin{equation*}
\begin{split}
2 \Psi_{2,\Upsilon}(f)(x) &= \sum_{y\in X\setminus\{x\}} l(y) \Big( |l|_1 \Upsilon'(f(y)-f(x)) (f(y)-f(x)) + l(x) \Upsilon(f(x)-f(y)) \\
& \qquad\qquad \quad - \big(|l|_1 - l(x)\big)\Upsilon(f(y)-f(x)) \Big) + \sum_{\substack{y,z \in X \\ y,z \neq x}}l(y)l(z) \Upsilon(f(z)-f(y)).
\end{split}
\end{equation*}
By positivity of the mapping $r \mapsto \Upsilon(r)$ and \eqref{eq:Palmesqaureestimate}, we proceed as follows
\begin{align*}
2\Psi_{2,\Upsilon}(f)(x) &\geq  2\kappa \Psi_\Upsilon(f)(x) \\
&\quad+ \sum_{y\in X\setminus\{x\}} l(y) \Big( |l|_1 \Upsilon'(f(y)-f(x)) (f(y)-f(x)) + l(x) \Upsilon(f(x)-f(y)) \\
& \qquad\qquad \quad - \big(|l|_1+2\kappa - l(x)\big)\Upsilon(f(y)-f(x)) \Big)\\
&\geq 2\kappa \Psi_\Upsilon(f)(x) + 2 \alpha c_\delta \sum_{y \in X \setminus \{x\}} l(y) \big|f(x)-f(y)\big|^{1+\delta} \\
&\quad + \sum_{y \in X\setminus \{x\}}l(y) \Big( |l|_1 \Upsilon'(f(y)-f(x)) (f(y)-f(x)) + (l(x)-2\alpha) \Upsilon(f(x)-f(y)) \\
&\qquad\qquad\quad - (|l|_1+2\kappa -(l(x)-2\alpha)) \Upsilon(f(y)-f(x))\Big)\\
&= 2\kappa \Psi_\Upsilon(f)(x) + 2\alpha  c_\delta \sum_{y \in X\setminus \{x\}} l(y) \big|f(x)-f(y)\big|^{1+\delta} \\
&\quad + (l(x)-2\alpha)\sum_{y \in X \setminus \{x\}} l(y) \nu_{\frac{|l|_1}{l(x)-2\alpha},\frac{|l|_1+2\kappa}{l(x)-2\alpha}-1}(f(y)-f(x)),
\end{align*}
for $\kappa>0$, $\delta\geq 1$ and $\alpha>0$ such that $l(x)>2\alpha$. Now, we aim to apply Proposition \ref{gammacriterion} with $\gamma(r)=|r|^{1+\delta}$. Therefore it is desirable that the last summand of the latter term is non-negative. For this purpose we  employ \cite[Lemma A.4(ii)]{WZ}, which yields that if $|l|_1 \geq 2(l(x)-2\alpha)$ it suffices that
\begin{equation*}
\frac{2\kappa}{l(x)-2\alpha}-1 \leq 2^{\frac{3}{2}}\sqrt{\frac{|l|_1}{l(x)-2\alpha}}-1,
\end{equation*}
which is equivalent to $\kappa\leq \sqrt{2|l|_1(l(x)-2\alpha)}$. In case that $|l|_1<2 (l(x)-2\alpha)$ it suffices by \cite[Lemma A.4(ii)]{WZ} that
\begin{equation*}
\frac{2\kappa}{l(x)-2\alpha}-1 \leq \frac{2|l|_1}{l(x)-2\alpha}-1,
\end{equation*}
which is equivalent to $\kappa\leq |l|_1$. Clearly, $|l|_1\geq \sqrt{2 |l|_1(l(x)-2\alpha)}$ for at least one $x \in X$. Consequently, we obtain by Proposition \ref{gammacriterion} that $L$ satisfies $CD_\Upsilon\big(\sqrt{2|l|_1(l_*-2\alpha)},F\big)$  with  $CD$-function $F(r)= \frac{\alpha c_\delta}{|l|_1^\delta} r^{1+\delta}$, $r\geq 0$, for any $\alpha \in \big(0,\frac{l_*}{2}\big)$ and any $\delta\geq 1$.
\end{example}
\begin{example}\label{ex:Kn}
We choose $l(x)=1$ in the setting of Example \ref{ex:keql} for any $x\in X$ with given finite state space $X$ consisting of $n$ elements, $n\geq 2$. Then the underlying graph to $L$ is given by the complete graph $K_n$ and $L$ satisfies $CD_\Upsilon(\sqrt{2n(1-2\alpha)},\frac{n}{\alpha})$ for any $\alpha\in(0,\frac{1}{2})$ by Example \ref{ex:keql} in the case of $\delta=1$ (recall that we have $c_1=1$ in \eqref{eq:Palmesqaureestimate}). It is natural to ask whether there also exists a dimension bound which is uniform in $n$ while having non-negative curvature as it is the case for the Bakry-\'Emery condition (see e.g. \cite{JoLi}). Interestingly, in Example \ref{ex:nouniformdimensionforKn} we are able to give a negative answer to this question.
\end{example}

The procedure described in Example \ref{ex:keql} can be seen as a guidance for other examples where the mapping $\nu_{c,d}$ plays a similar role,  e.g. for weighted 4-cycles, finite birth-death processes and weighted stars as discussed in \cite{WZ}.

The case of (R)-Ricci-flat graphs will be discussed seperatly below.  For the reader's convenience we recall the definition of (R)-Ricci flat graphs, which originates from the work of \cite{CKKLP}.
\begin{defi}\label{def:RRicci}
Let $G=(V,E)$ be an unweighted $d$-regular graph. We call $G$ (R)-Ricci-flat at $x\in V$ if there exist maps $\eta_i: B_1(x) \to V$ (where $B_1(x)$ denotes the closed ball with radius $1$ and center $x$ with respect to the combinatorical graph distance) for $1 \leq i \leq d$ satisfying the following properties:
\begin{itemize}
\item[(i)] $\eta_i(u)\in B_1(u) \setminus \{u\} $ for any $u \in B_1(x)$ and $i \in \{1,...,d\}$,
\item[(ii)] $\eta_i(u) \neq \eta_j(u)$, whenever $i \neq j$,
\item[(iii)] $\bigcup_j \eta_j(\eta_i(x)) = \bigcup_j \eta_i(\eta_j(x))$ for any $i \in \{1,...,d\}$,
\item[(iv)] $\eta_i(\eta_i(x))=x$ for any $i \in \{1,...,d\}$.
\end{itemize} 
We call $G$ (R)-Ricci-flat if $G$ is (R)-Ricci-flat at each $x \in V$.
\end{defi}
(R)-Ricci-flat graphs constitute a subclass of Ricci-flat graphs with Bakry-\'Emery condition $CD(2,\infty)$, see \cite{CKKLP}. Important examples are given by complete bipartite graphs and, since (R)-Ricci-flat graphs are invariant under tensorization, by the hypercube, cf. \cite{CKKLP}. In \cite[Example 5.12]{WZ} it has been shown that a Markov generator with underlying graph being (R)-Ricci-flat  even satisfies $CD_\Upsilon(2,\infty)$.

\begin{example}\label{ex:RRicci}
Let the transition rates be given such that the underlying graph to $L$ is  a ($d$-regular) $(R)$-Ricci-flat graph with vertex set $X$. Further, let $x \in X$ be chosen arbitrary and let $\big( \eta_i \big)_{i=1,...,d}$ denote the corresponding mappings from Definition \ref{def:RRicci}.   In \cite[Example 5.12]{WZ} it has been shown that for $f \in \R^X$ the estimate 
\begin{equation*}
2\Psi_{2,\Upsilon}(f)(x) \geq 4 \Psi_\Upsilon (f)(x) + \sum_{i=1}^d \nu_{2,5}(f(\eta_i(x))-f(x))
\end{equation*}
holds true. Now, one readily checks (see also the proof of  \cite[Lemma A.3]{WZ}) that $\nu_{2,5}$ is strictly convex on $\R$. We infer from Proposition \ref{gammacriterion} that $L$ satisfies $CD_\Upsilon(2,F)$ with $F(r)=\frac{d}{2}\nu_{2,5}(-\frac{r}{d})$, $r\geq 0$, which is a $CD$-function due to the  asymptotic behavior at $\infty$ and by Remark \ref{remarkCDfunctions}. Interestingly, since the curvature constant in $CD_\Upsilon(2,\infty)$ is optimal in general, which follows, for instance, from the case where the underlying graph to $L$ is given by the hypercube (cf. \cite{WZ}), we do not need a trade off from the curvature constant in order to achieve the $CD_\Upsilon(2,F)$ condition (which is, for instance, in contrast to the procedure described in Example \ref{ex:keql}). Further, note that $\nu_{2,5}$ behaves only quartic near $0$. In fact,  there does not exist in general a $CD$-function $\hat{F}$ behaving quadratically near zero such that $L$ satisfies $CD_\Upsilon(2,\hat{F})$, by combining Remark \ref{CDdiscussionremark}(ii) with the fact that  the hypercube does not satisfy $CD(2,m)$ for  some $m<\infty$ (see \cite{CLP}).
\end{example}

\begin{example}\label{ex:birthdeath}
Here we consider a birth-death process  with infinite state space $X= \N_0$. We use the notation originating from \cite{CDP}, which has also been used in \cite{WZ}, and introduce the functions $a,b: X \to [0,\infty)$ with $a(x)=k(x,x+1)$, $b(x)=k(x,x-1)$, $b(0)=0$, $b(x)>0$ otherwise, and $a(x)>0$ for any $x\in X$. Moreover, we set  $k(x,y)=0$ whenever $|x-y| >1$. The detailed balance condition now reads as
\begin{equation}\label{eq:dbforbd}
a(x)\pi(x)= b(x+1)\pi(x+1)
\end{equation}
for any $x \in X$. Note that the measure $\mu$ given by $\mathrm{d}\mu=\pi \mathrm{d}\#$ is a finite measure if and only if
\begin{equation*}
\sum\limits_{x=1}^\infty \frac{a(x-1)\cdot \cdot \cdot a(0)}{b(x)\cdot \cdot \cdot b(1)} < \infty.
\end{equation*}
We assume monotonicity of the rates in the sense that $a(x) \leq a(x+1)$ and $b(x+1) \geq b(x)$ for any $x \in X$ and moreover that
\begin{equation}\label{eq:Caputoassumption}
a(x)-a(x+1) + b(x+1)-b(x) \geq \kappa
\end{equation}
holds for any $x \in X$ and some $\kappa>0$. Those assumptions led to modified logarithmic Sobolev inequalities in \cite{CDP}, and in \cite{WZ} it has been shown that they imply $CD(\frac{\kappa}{2},\infty)$. Apparently they also entail that $a(x)\leq a(0)$ for any $x \in X$ and  $b(x) \to \infty$ as $x \to \infty$.

We specify for $x\in X$ a function $f_x\in \R^X$ such that   $f_x(x+2)= f_x(x+1) +t$ and $f_x(x-2)=f_x(x-1)+s$, where we define $t= f_x(x+1)-f_x(x)$ and $s=f_x(x-1)-f_x(x)$ (this is called  minimizing $\Psi_{2,\Upsilon}(f)$ over the second neighborhood throughout \cite{WZ}) and  set $t=0$. Then we observe from \eqref{eq:Psi2Formel} (see also the  representation formula for $\Psi_{2,\Upsilon}(f)(x)$ that has been established in \cite[Example 5.13]{WZ}) 
\begin{equation*}
2\Psi_{2,\Upsilon}(f_x)(x)=b(x) \Big( \Upsilon(s) \big(b(x-1) - b(x) -a(x)\big) + \Upsilon(-s) a(x-1) + \Upsilon'(s)s \big(a(x-1)+ b(x) - b(x-1)\big) \Big).
\end{equation*}
Assuming that  $2\Psi_{2,\Upsilon}(f_x)(x)$ is greater than or equal to $\frac{1}{n} b(x)^2 s^2$ (which equals $\frac{1}{n} (-L f_x (x) )^2$) for some $n \in [1,\infty)$ implies that 
\begin{align*}
0&\leq b(x) \big( \Upsilon'(s) s - \Upsilon(s) - \frac{1}{n} s^2 \big) + b(x-1) \big( \Upsilon(s) - \Upsilon'(s)s \big)\\&\quad + a(x-1) \big( \Upsilon(-s)+\Upsilon'(s)s\big) - a(x)\Upsilon(s)\\
&\leq a(0) \big( \Upsilon(-s) + \Upsilon'(s)s \big)+ b(x) \big( \Upsilon'(s) s - \Upsilon(s) - \frac{1}{n} s^2 \big)
\end{align*}
Choosing $s<0$ such that $\Upsilon'(s)s - \Upsilon(s) - \frac{1}{n} s^2 <0$ and sending $x \to \infty$ yields a contradiction.

Interestingly, in \cite{WZ} it has been shown that under an  assumption which is stronger than \eqref{eq:Caputoassumption} the $CD_\Upsilon(\kappa_0,\infty)$ condition holds with some positive constant $\kappa_0>0$. This shows that it is possible to have positive curvature bounds while having no finite dimension bound regarding the $CD_\Upsilon(\kappa,n)$ condition for some $\kappa>0$. Furthermore, note that we will show by means of Corollary \ref{finitespaceforhigherpowers} below that a $CD_\Upsilon$ condition with positive curvature bound and  a non-quadratic power-type $CD$-function does not hold either.
\end{example}
Next, we give a quite simple negative criterion for the existence of a dimension term with regard to the quadratic $CD$-function in the infinite state space case.
\begin{proposition}\label{prop:negativecriterionfordimension}
If there exists a sequence $(x_m)_{m \in \N}\subset X$ such that
\begin{equation}\label{necconditionwithNandM1}
\frac{N(x_m)}{(M_1(x_m))^2}\to 0
\end{equation}
as $m \to \infty$, then  there does not exist some $n<\infty$ such that $CD_\Upsilon(0,n)$ holds true.
\end{proposition}
\begin{proof}
We consider $\big(f_m\big)_{m \in \N} \subset \ell^\infty(X)$ given by  $f_m(y) = t$ for any $y\in X \setminus \{x_m\}$ and $f_m(x_m)=0$, $m \in \N$, with $t <0$ being specified below. We read from \eqref{eq:Psi2Formel} that
\begin{align*}
2 \Psi_{2,\Upsilon}(f_m)(x_m) = N(x_m) \big( \Upsilon'(t)t + \Upsilon(-t)\big)+ (M_1(x_m))^2 \big(\Upsilon'(t)t - \Upsilon(t)\big)
\end{align*} 
holds. Moreover, we have $(-L f_m(x_m))^2= t^2 (M_1(x_m))^2$. Thus, $2\Psi_{2,\Upsilon}(f_m)(x_m) \geq \frac{1}{n} (-L f_m (x_m))^2$, with $1\leq n<\infty$, is equivalent to 
\begin{equation}\label{NxnM1squareestimate}
\frac{N(x_m)}{(M_1(x_m))^2}  \big( \Upsilon'(t)t + \Upsilon(-t)\big) + \Upsilon'(t)t - \Upsilon(t) - \frac{1}{n} t^2 \geq 0.
\end{equation}
Choosing $t<0$ such that $\Upsilon'(t)t-\Upsilon(t) - \frac{1}{n} t^2 < 0$ and sending $m\to \infty$, yields by \eqref{necconditionwithNandM1} a contradiction to \eqref{NxnM1squareestimate}.
\end{proof}
Clearly, Proposition \ref{prop:negativecriterionfordimension} also yields a necessary condition for families of Markov generators satisfying a uniform $CD_\Upsilon(0,n)$ condition. For the sake of clarity we will state this in the following corollary in the case that the underlying graphs to the corresponding Markov generators are unweighted, in which case the mappings  $N$ and $M_1$ are equal. The following corollary follows from the same arguments as in the proof of Proposition \ref{prop:negativecriterionfordimension}. Despite its simplicity, these findings lead to a remarkable difference between the $CD_\Upsilon(0,n)$ and the $CD(0,n)$ condition, as it will be demonstrated by Example  \ref{ex:nouniformdimensionforKn}.
\begin{corollary}\label{graphfamilycorollary}
Let $I$ be an arbitrary index set and $(L_i)_{i \in I}$ a family of Markov generators whose respective underlying graphs are unweighted and with corresponding state space $(X_i)_{i\in I}$. Assume that there exist sequences $(i_m)_{m \in \N}\subset I$ and $(x_m)_{m \in \N} \subset X$, where $X= \bigcup_{i \in I} X_i$, such that $M_1^{(i_m)}(x_m)\to \infty$ as $m \to \infty$. Here the upper index denotes that the function $M_1$ corresponds to the respective Markov generator. Then there exists no $n<\infty$ such that $L_i$ satisfies $CD_\Upsilon(0,n)$ for all $i \in I$.

In particular, if the underlying graph to a Markov generator $L$ is given by an unweighted graph with unbounded vertex degree, then there exists no $n<\infty$ such that $CD_\Upsilon(0,n)$ holds true.
\end{corollary}
\begin{example}\label{ex:nouniformdimensionforKn}
We consider the index set $I=\{n \in \N: n\geq 2\}$ (in the sense of Corollary \ref{graphfamilycorollary}) and  the Markov generator $L_n$ whose underlying graph is given by the complete graph $K_n$ for any $n \in I$. It is known on the one hand that $CD(0,4)$ holds for any $L_n$ (see  \cite[Proposition 3]{JoLi}), i.e. a dimension-term exists under non-negative curvature with respect to the Bakry-\'Emery CD-condition  that is uniform with regard to $n$. On the other hand, due to Corollary \ref{graphfamilycorollary} there does not exist a (uniform) $d<\infty$ such that $CD_\Upsilon(0,d)$ holds for any $L_n$, $n \in I$. 
\end{example}

\section{Entropy-Information inequalities}\label{sec:EIinequalities}
From now on we assume that the unique and reversible invariant measure $\mu$ is a probability measure. 

We consider the entropy
\begin{equation}\label{eq:Entropy}
\mathrm{Ent}_\mu(f)= \int_X f \log f \mathrm{d}\mu - \int_X f \mathrm{d}\mu\, \log \int_X f \mathrm{d}\mu
\end{equation}
for any positive function $f \in \ell^1(\mu)$. It is well known that $\mathrm{Ent}_\mu(f)\geq 0$. Note that we also allow for the value of $\mathrm{Ent}_\mu(f)=\infty$. 

The Fisher information is given by
\begin{equation}\label{eq:FisherInfo}
\mathcal{I}(f) = \frac{1}{2}\sum_{x,y \in X}k(x,y)\big(f(y)-f(x)\big)\big(\log f(y) - \log f(x)\big)\pi(x).
\end{equation}
If we assume that $M_1\in \ell^1(\mu)$ then $f\in \ell^{\infty,+}(X)$ ensures that $\mathcal{I}(f)<\infty$. Further, in the latter case we have the representation  
\begin{equation*}
\mathcal{I}(f)= \int_X f \Psi_\Upsilon(\log f) \mathrm{d}\mu,
\end{equation*}
see \cite[Section 3]{WZ}, where the formula has been established for $f\in \ell^{\infty,+}(X)$ being a probability densitiy with respect to $\mu$, although the proof extends verbatim to the  general case. Since we sum up in the right-hand side of \eqref{eq:FisherInfo} over non-negative entries, we can extend the functional $\mathcal{I}$ to positive functions $f:X\to(0,\infty)$, where we allow for the value of $\mathcal{I}(f)=\infty$. 

Note that the assumption of the Markov chain being irreducible implies that $\mathcal{I}(f)$=0 if and only if $f$ is constant and positive. 

Besides that, one readily verifies that the well known scaling behavior
\begin{equation}\label{eq:scalingEntro}
\mathrm{Ent}_\mu(cf)=c\, \mathrm{Ent}_\mu(f)
\end{equation}
and 
\begin{equation}\label{eq:scalingFisher}
\mathcal{I}(cf) = c\, \mathcal{I}(f)
\end{equation}
holds respectively for any constant $c>0$.

Our main object of investigation in the remaining part of this article will be the following family of functional inequalities.
\begin{defi}\label{definitionEIPhi}
We say that $L$ satisfies an entropy-information inequality $EI(\Phi)$ with respect to a strictly increasing and concave $C^1$-function $\Phi:(0,\infty)\to (0,\infty)$, which we refer to as the growth function, if for every $f \in P_*(X)$ with $\mathrm{Ent}_\mu(f)<\infty$ and $\mathcal{I}(f)<\infty$
\begin{equation}\label{EntropieInformationsInequality}
\mathrm{Ent}_\mu(f)\leq \Phi(\mathcal{I}(f))
\end{equation}
holds, where we set $\Phi(0):=\lim\limits_{r \to 0^+}\Phi(r)$.
\end{defi}
A well known example for an entropy-information inequality is the modified logarithmic Sobolev inequality 
\begin{equation}\label{eq:mLSI}
\mathrm{Ent}_\mu(f) \leq \frac{1}{2\kappa}\mathcal{I}(f)
\end{equation}
with constant $\kappa>0$. See \cite{BoTe} for an extensive account on modified logarithmic Sobolev inequalities in the discrete setting of Markov chains. Further, the functional inequality \eqref{eq:mLSI} was subject of investigation in \cite{CDP}, \cite{ErFa}, \cite{ErMa} and \cite{FaMa}, as well as in \cite{WZ} where it has been shown that $CD_\Upsilon(\kappa,\infty)$ (with $\kappa>0$) together with positive recurrence and the integrability conditions $M_1\in \ell^2(\mu)$ and $M_2\in \ell^1(\mu)$ imply \eqref{eq:mLSI} with constant $\kappa$.

\begin{remark}(i) The diffusive counterpart to Definition \ref{definitionEIPhi}, so called entropy-energy inequalities, are defined for growth functions mapping to $\R$ instead of $(0,\infty)$, see \cite{BGL}. This generality in the context of \cite{BGL} allows to include the quite important special case of the Euclidean logarithmic Sobolev inequality (cf. \cite[Proposition 6.2.5]{BGL}). However, assuming that $\Phi$ is non-negative is not a restriction in our setting where we have supposed that $\mu$ is a probability measure. Indeed, applying \eqref{EntropieInformationsInequality} to $f=\mathds{1}$ shows that $\lim\limits_{r \to 0^+} \Phi(r)<0$ is impossible to hold.

(ii) It will turn out to be quite useful to write \eqref{EntropieInformationsInequality} in an equivalent linearized form. More precisely, as $\Phi$ is concave we deduce from $\Phi(s) \leq \Phi(r)+ \Phi'(r)(s-r)$, $s,r \in (0,\infty)$, that \eqref{EntropieInformationsInequality} implies
\begin{equation}\label{eq:EEIlinearizedform}
\mathrm{Ent}_\mu(f) \leq \Phi'(r)\, \mathcal{I}(f)+ \Theta(r)
\end{equation}
for any $r \in (0,\infty)$, where $\Theta(r)=\Phi(r)-\Phi'(r)r$. Conversely, specifying $r=\mathcal{I}(f)$, \eqref{eq:EEIlinearizedform} implies \eqref{EntropieInformationsInequality}. Note that concavity also implies that $\Theta(r)\geq 0$ for any $r>0$ since $\lim\limits_{r \to 0^+}\Phi(r) \geq 0$. In particular, this implies that \eqref{eq:EEIlinearizedform} also holds for the case of $\mathcal{I}(f)=0$, or equivalently for $f=\mathds{1}$.

(iii) Let $f\in \ell^1(\mu)$ be positive with $\mathrm{Ent}_\mu(f)<\infty$ and $\mathcal{I}(f)<\infty$. The entropy-information inequality in the form \eqref{eq:EEIlinearizedform} extends to $f$ by 
\begin{equation}\label{eq:EIPhinonnormalized}
\mathrm{Ent}_\mu(f)\leq \Phi'(r) \mathcal{I}(f) + \Theta(r) \int_X f d\mu
\end{equation}
for any $r \in (0,\infty)$. This is a consequence of the scaling behavior  \eqref{eq:scalingEntro} and \eqref{eq:scalingFisher}, after having applied \eqref{eq:EEIlinearizedform} to $\frac{f}{\Vert f \Vert_1}$.
\end{remark}

Now, we come to the main theorem of this section, that links the previous section to the notion of entropy-information inequalities.
\begin{theorem}\label{maininequalityunderCD}
Let $M_1 \in \ell^2(\mu)$, $M_2 \in \ell^1(\mu)$ and the Markov chain generated by $L$ be positive recurrent. Further, let $L$ satisfy $CD_\Upsilon(\kappa,F)$ with $\kappa>0$ and a convex $CD$-function $F: [0,\infty) \to [0,\infty)$ such that $\left.F\right|_{(0,\infty)}\in C^1\big((0,\infty)\big)$ and $\frac{F'(r)r}{F(r)}\geq 1+\delta$ holds for any $r > 0$ and some $\delta>0$. Let $G:(0,\infty) \to (0,\infty)$ denote the inverse function of $r \mapsto \frac{F(r)}{r}$, $r>0$.  Then 
\begin{equation}\label{eq:mainresultinequality}
\mathrm{Ent}_\mu(f) \leq \int_0^\infty G \Big( \frac{\kappa}{e^{2\delta\kappa t} \big( 1+ \frac{\kappa \mathcal{I}(f)}{F(\mathcal{I}(f))}\big) - 1}\Big)\mathrm{d}t
\end{equation}
holds for any $f \in P_*(X)$ with $\mathrm{Ent}_\mu(f)<\infty$ and $\mathcal{I}(f)\in (0,\infty)$.
\end{theorem}
\begin{proof}
It suffices to deduce the claim for $f \in P_*^+(X)$. The full statement follows then from the same standard truncation argument as presented in \cite[Lemma 3.2]{WZ} and the dominated convergence theorem for approximating the right-hand side of \eqref{eq:mainresultinequality}.
For $f \in P_*^+(X)$ we set $\Lambda(t)= \mathrm{Ent}_\mu(P_t f)$, $t \geq 0$. It is well known that 
\begin{equation}\label{eq:firstDerivative}
\Lambda'(t) = -\mathcal{I}(P_t f)
\end{equation}
is valid provided that $M_1 \in \ell^1(\mu)$. Further, we infer from \cite[Theorem 3.4]{WZ} that
\begin{equation}\label{eq:secondDerivative}
\Lambda''(t) = 2 \int_X P_t f \Psi_{2,\Upsilon}(\log P_t f)\mathrm{d}\mu
\end{equation}
holds true given the assumptions $M_1 \in \ell^2(\mu)$ and $M_2 \in \ell^1(\mu)$. We apply $CD_\Upsilon(\kappa,F)$ to deduce
\begin{align*}
\Lambda''(t) &\geq -2\kappa \Lambda'(t) + 2 \int_X P_t f \, F_0(- L (\log P_t f)) \mathrm{d}\mu \\
&\geq -2\kappa \Lambda'(t) +2 F_0 \Big(- \int_X P_t f L(\log P_t f) \mathrm{d}\mu\Big),
\end{align*}
where the latter follows from convexity of the trivial extension $F_0$ (which follows from convexity of $F$), the fact that $P_t f$ is a probability density with respect to $\mu$, which follows from $\mu$ being invariant for $(P_t)_{t \geq 0}$, and Jensen's inequality.
Now, by the  identity \eqref{eq:FFI} (cf. \cite[Lemma 2.2]{WZ}) and $\mu$ being invariant, we have
\begin{align*}
-\int_X P_t f L(\log P_t f) \mathrm{d}\mu =  \int_X P_t f\, \Psi_\Upsilon(\log P_t f) \mathrm{d}\mu = \mathcal{I}(P_t f) 
\end{align*}
and hence we end up with the differential inequality
\begin{equation}\label{eq:diffineqformEE}
\Lambda''(t) \geq - 2\kappa \Lambda'(t) + 2 F(-\Lambda'(t)).
\end{equation}
Note that in fact $\Lambda'(t)<0$ holds, since we have $f=\mathds{1}$ otherwise.
Further, we observe that
\begin{align*}
\frac{d}{dt} \Big[e^{-2\delta\kappa t} \Big( 1-\kappa &\frac{\Lambda'(t)}{F(-\Lambda'(t))}\Big)\Big] = e^{-2\delta \kappa t} \Big( \frac{2\delta\kappa^2 \Lambda'(t)}{F(-\Lambda'(t))} - 2\delta\kappa - \kappa\Lambda''(t)\frac{F(-\Lambda'(t))+ F'(-\Lambda'(t))\Lambda'(t)}{F(-\Lambda'(t))^2}\Big)\\
&= \frac{\kappa e^{-2\delta\kappa t}}{F(-\Lambda'(t))} \Big(  2 \delta\kappa \Lambda'(t) - 2 \delta F(-\Lambda'(t)) - \Lambda''(t) \Big( 1+ \frac{F'(-\Lambda'(t))\Lambda'(t)}{F(-\Lambda'(t))}\Big)\Big)\\
&\geq \frac{\delta\kappa e^{-2\delta\kappa t}}{F(-\Lambda'(t))} \Big( 2 \kappa \Lambda'(t) - 2 F(-\Lambda'(t)) + \Lambda''(t)\Big),
\end{align*}
where we have applied the condition $\frac{-\Lambda'(t)F'(-\Lambda'(t))}{F(-\Lambda'(t))}\geq 1+\delta$ and $\Lambda''(t)\geq 0$ in the last step. Hence, \eqref{eq:diffineqformEE} yields that the  mapping $t \mapsto e^{-2\delta\kappa t} \big( 1-\kappa \frac{\Lambda'(t)}{F(-\Lambda'(t))}\big)$ is increasing. In particular, this implies that
\begin{align*}
e^{-2\delta\kappa t} \Big(1- \kappa \frac{\Lambda'(t)}{F(-\Lambda'(t))}\Big) \geq 1 - \kappa \frac{\Lambda'(0)}{F(-\Lambda'(0))},
\end{align*}
which can be rearranged to
\begin{equation*}
\frac{F(-\Lambda'(t))}{-\Lambda'(t)} \leq \frac{\kappa}{e^{2\delta\kappa t} \big(1-\kappa \frac{\Lambda'(0)}{F(-\Lambda'(0))}\big) - 1}.
\end{equation*}
Then, using $\mathcal{I}(f) = - \Lambda'(0)$, we obtain
\begin{equation*}
-\Lambda'(t) \leq G \Big( \frac{\kappa}{e^{2\delta\kappa t} \big( 1+ \frac{\kappa\mathcal{I}(f)}{F(\mathcal{I}(f)}\big) - 1}\Big).
\end{equation*}
Consequently, we conclude
\begin{equation}\label{eq:WithTmaininequality}
\Lambda(0)-\Lambda(T) = - \int_0^T \Lambda'(t) \mathrm{d}t \leq \int_0^T G \Big( \frac{\kappa}{e^{2\delta\kappa t} \big( 1+ \frac{\mathcal{I}(f)}{F(\mathcal{I}(f)}\big) - 1}\Big)\mathrm{d}t.
\end{equation}
The claim follows by sending $T\to \infty$. Indeed, $\Lambda(T)\to 0$ as $T \to \infty$ follows from the dominated convergence theorem, the Markov chain being ergodic and $(P_t)_{t \geq 0}$ being a Markov semigroup.
\end{proof}
\begin{remark}\label{rem:CDfunctionassumption}(i)
The crucial assumption that
\begin{equation}\label{eq:CDfunctionassumption2}
F'(r)\geq \frac{(1+\delta)F(r)}{r}, 
\end{equation}
holds for some $\delta>0$ and any $r>0$
implies by Gronwall's inequality that we have for  fixed $a>0$
\begin{equation*}
\frac{F(r)}{r^{1+\delta}} \geq \frac{F(a)}{a^{1+\delta}},
\end{equation*}
for any $r > a>0$, i.e. the mapping $r \mapsto \frac{F(r)}{r^{1+\delta}}$, $r>0$, is increasing. Conversely, differentiating $r \mapsto \frac{F(r)}{r^{1+\delta}}$, $r>0$, the property \eqref{eq:CDfunctionassumption2} follows provided that $r\mapsto \frac{F(r)}{r^{1+\delta}}$, $r>0$, is increasing. Hence, both properties are equivalent. In particular, we observe that \eqref{eq:CDfunctionassumption2} ensures that $F$ grows at least like $r^{1+\delta}$ as $r\to\infty$. On the other hand, recall that we have seen in Remark \ref{CDdiscussionremark}(iii) that $F$ can not behave better than quadratic at $0$ provided that $CD_\Upsilon(\kappa,F)$ holds for some $\kappa \in \R$.

(ii) Note that the mapping $G:(0,\infty)\to(0,\infty)$ is in fact well defined, since assuming that $r\mapsto \frac{F(r)}{r^{1+\delta}}$ is increasing on $(0,\infty)$ for some $\delta>0$ implies that $\frac{F(r)}{r}\to 0$ as $r\to 0$ and $\frac{F(r)}{r}\to \infty$ as $r\to \infty$.

(iii) We interpret the integral on the right-hand side of \eqref{eq:mainresultinequality} as $\infty$ in the case that the integral is divergent. In fact, this situation appears even under the assumptions of Theorem \ref{maininequalityunderCD} as the  $CD$-function $F(r)= r^me^{-\frac{1}{r^m}}$, $r\geq 0$, for some $m>1$ shows. Indeed, one readily verifies that \eqref{eq:CDfunctionassumption2} with $\delta = m-1$ and convexity of $F$ respectively hold true. The problem results from the behavior of $F$ in the origin. More precisely,  $\frac{F(r)}{r}$ converges faster to $0$ than $e^{-\frac{1}{r^m}}$ as $r \to 0^+$, which yields that $G(r)$ tends slower to $0$ than $(\log \frac{1}{r})^{-\frac{1}{m}}$ as $r \to 0^+$. Consequently the integrand in the right-hand side of \eqref{eq:mainresultinequality} dominates a behavior of $t^{-\frac{1}{m}}$ as $t\to \infty$, which yields that the integral does not converge. 
\end{remark}
For general $CD$-functions the integral on the right-hand side of \eqref{eq:mainresultinequality}  can not be calculated explicitly and, moreover, it is not clear whether the mapping $r \mapsto \int_0^\infty G \big(\frac{\kappa}{e^{2\kappa t}(1+\frac{\kappa r}{F(r)})-1}\big) \mathrm{d}t$ is concave. We will focus in the sequel on the situation where the $CD$-function is given by some power-type function, in which case the mapping $G:(0,\infty)\to(0,\infty)$ of Theorem \ref{maininequalityunderCD} can be given explicitly. In fact, the following result shows in particular  that the functional inequality \eqref{eq:mainresultinequality} is compatible with Definition \ref{definitionEIPhi} for power-type $CD$-functions.
\begin{proposition}\label{prop:EIforpowertype}
Let $M_1 \in \ell^2(\mu)$, $M_2 \in \ell^1(\mu)$ and the Markov chain generated by $L$ be positive recurrent. Further, let $L$ satisfy $CD_\Upsilon(\kappa,F)$ with $\kappa>0$ and $F(r)=\frac{r^{\delta+1}}{n}$, $r\geq 0$, for some $\delta\geq 1$ and $n \in (0,\infty)$. Then $L$ satisfies $EI(\Phi)$ with the growth function 
\begin{equation}\label{eq:powertypegrowthfunction}
\Phi(r) = \frac{\sqrt[\delta]{\kappa n}}{2\kappa}\int_{\frac{\sqrt[\delta]{\kappa n}}{r}}^\infty \frac{v^{\delta-2}}{v^\delta + 1}\mathrm{d}v, \, r>0.
\end{equation}
Moreover, the growth function $\Phi$ satisfies the following assertions:
\begin{itemize}
\item[(i)] $\Phi$ is bounded if and only if $\delta>1$,
\item[(ii)] $\int_0^\infty \frac{\Phi(s^2)}{s^2}\mathrm{d}s <\infty$.
\end{itemize}
\end{proposition}
\begin{proof}
Clearly, $F$ is convex and $F'(r)r=(1+\delta)F(r)$ holds for any $r>0$. The mapping  $G:(0,\infty)\to(0,\infty)$ from Theorem \ref{maininequalityunderCD} is given by $G(r)=\sqrt[\delta]{n r}$, $r>0$. Now, it follows from elementary substitution that
\begin{equation*}
\int_0^\infty \sqrt[\delta]{\frac{\kappa n}{e^{2\delta\kappa t}(1+\frac{\kappa n}{r^\delta})-1}}\mathrm{d}t =\frac{\sqrt[\delta]{\kappa n}}{2\delta\kappa}\int_{\frac{\kappa n}{r^{\delta}}}^\infty \frac{1}{\sqrt[\delta]{u}(u+1)}\mathrm{d}u = \frac{\sqrt[\delta]{\kappa n}}{2\kappa} \int_{\frac{\sqrt[\delta]{\kappa n}}{r}}^\infty \frac{v^{\delta-2}}{v^\delta + 1}\mathrm{d}v= \Phi(r).
\end{equation*}
By means of Definition \ref{definitionEIPhi} and Theorem \ref{maininequalityunderCD} it suffices to prove that $\Phi$ is concave in order to deduce that $L$ satisfies $EI(\Phi)$ (note that there is nothing to show for the case of $\mathcal{I}(f)=0$). To that aim, we differentiate $\Phi$ and observe
\begin{align}
\Phi'(r) &= \frac{n}{2(\kappa n + r^\delta)},\label{eq:derivativepowertypegrowth}\\
\Phi''(r) &= - \frac{n\delta r^{\delta-1}}{2(\kappa n + r^\delta)^2},\nonumber
\end{align}
which implies concavity of $\Phi$. 

The growth function $\Phi$ is bounded if and only if the integral in the right-hand side of \eqref{eq:powertypegrowthfunction} converges as $r\to \infty$. The latter property holds true if and only if  the integral  $
\int_{\frac{\sqrt[\delta]{\kappa n}}{r}}^1 v^{\delta-2} \mathrm{d}v $ converges as $r\to \infty$, which happens to be true if and only if $\delta>1$. 

Regarding the remaining assertion, we note that there is nothing to show for the behavior at $\infty$ by boundedness of $\Phi$ in case of $\delta>1$ and by the explicit formula for the growth function in the special case of $\delta=1$, which will be deduced in Corollary \ref{logEIunderCD(kappa,n)} below. As to the behavior at $0$, we observe for $\varepsilon>0$ that
\begin{align*}
\int_0^\varepsilon\frac{1}{s^2}\int_{\frac{\sqrt[\delta]{\kappa n}}{s^2}}^\infty \frac{v^{\delta-2}}{v^\delta + 1}\mathrm{d}v\,\mathrm{d}s \leq \int_0^\varepsilon \frac{1}{s^2}\int_{\frac{\sqrt[\delta]{\kappa n}}{s^2}}^\infty \frac{1}{v^2}\mathrm{d}v\,\mathrm{d}s = \frac{\varepsilon}{\sqrt[\delta]{\kappa n}}.
\end{align*}
\end{proof}
\begin{remark}
Let us emphasize an analogy between the relation of the curvature-dimension conditions $CD_\Upsilon(\kappa,F)$ (with power-type $CD$-function $F$) and $CD_\Upsilon(\kappa,\infty)$, and the resulting functional inequalities, $EI(\Phi)$ with $\Phi$ given by \eqref{eq:powertypegrowthfunction} and the modified logarithmic Sobolev inequality \eqref{eq:mLSI} with constant $\kappa>0$. Clearly, $CD_\Upsilon(\kappa,F)$ implies $CD_\Upsilon(\kappa,\infty)$. On the other hand, concavity of $\Phi$ (where $\Phi$ is given by \eqref{eq:powertypegrowthfunction}) implies that
$
\Phi(r)\leq \Phi(s)+\Phi'(s)(r-s), \, r,s\in (0,\infty).
$
Using identity \eqref{eq:derivativepowertypegrowth}, this estimate yields 
\begin{equation*}
\Phi(r)\leq \frac{r}{2\kappa}
\end{equation*} 
when sending $s\to 0$. Hence, $EI(\Phi)$ is stronger than the modified logarithmic Sobolev inequality \eqref{eq:mLSI} with constant $\kappa>0$, where the latter is a consequence of $CD_\Upsilon(\kappa,\infty)$ (cf. \cite{WZ}).
\end{remark}
In the case of the $CD_\Upsilon(\kappa,n)$ condition, Theorem \ref{maininequalityunderCD} yields the following important entropy-information inequality.
\begin{corollary}\label{logEIunderCD(kappa,n)}
Let $M_1 \in \ell^2(\mu)$, $M_2 \in \ell^1(\mu)$ and the Markov chain generated by $L$ be positive recurrent. Further, let $L$ satisfy $CD_\Upsilon(\kappa,n)$ with $\kappa>0$ and $n<\infty$. Then  $EI(\Phi)$ holds with $\Phi(r)= \frac{n}{2}\log \Big( 1 + \frac{r}{\kappa n}\Big)$, $r>0$, i.e.
\begin{equation}\label{eq:logarithmicEI}
\mathrm{Ent}_\mu(f) \leq \frac{n}{2}\log \Big( 1+ \frac{\mathcal{I}(f)}{\kappa n}\Big)
\end{equation}
holds for any $f \in P_*(X)$ with $\mathrm{Ent}_\mu(f)<\infty$.
\end{corollary}
\begin{proof}
Choosing $\delta=1$ in \eqref{eq:powertypegrowthfunction}, the claim follows from elementary calculations.
\end{proof}
As we already highlighted in the introduction, we emphasize that \eqref{eq:logarithmicEI} serves as a natural discrete analogue to the logarithmic entropy-energy inequality \eqref{eq:classicallogEE}, which plays an important role in the diffusive setting of \cite{BGL} (in which case it holds true provided that $CD(\kappa,n)$ is valid).

In the following example we consider one of the most important instances of a birth-death process from Example \ref{ex:birthdeath}.
\begin{example}\label{ex:poisson}
As a special case of a birth-death process from Example \ref{ex:birthdeath} (with the notation taken from there), we consider the Poisson case which is given by the choice $a(x)=\lambda$, where $\lambda>0$ is called the intensity rate, and $b(x)=x$, both for any $x \in \N_0$. The invariant and reversible measure is given by the denisity $\pi_\lambda(x)= \frac{\lambda^x}{x!}e^{-\lambda} $, $x \in \N_0$. In \cite[Example 5.13]{WZ} it has been shown that there does not exists some $\kappa>0$ such that $L$ satisfies $CD_\Upsilon(\kappa,\infty)$. However, it is known that the Poisson case of the birth-death prosess satisfies the modified logarithmic Sobolev inequality $EI(\Phi)$ with $\Phi(r)=r$, see e.g. \cite{CDP}. Here we show, that this is the best possible entropy-information inequality for the Poisson case in the sense that if $\Phi$ grows slower than linear as $r\to \infty$, then $L$ fails to satisfy $EI(\Phi)$.  

To that aim we repeat an argument that has been used in \cite{CDP} to show sharpness of the corresponding modified logarithmic Sobolev inequality. Indeed, we consider $f_k(x) = \frac{e^{kx}}{e^{\lambda(e^k - 1)}}$, which can be readily checked to be an element of $P_*^+(X)$ for any $k \in \N$. We have
\begin{align*}
\mathrm{Ent}_\mu(f_k)&=\frac{1}{e^{\lambda e^k }}\sum_{x \in \N_0} \frac{e^{kx} \lambda^x (kx - \lambda(e^k - 1))}{x!}\\
&= \frac{1}{e^{\lambda e^k}}\Big( k \sum_{x \in \N}\frac{(e^k \lambda)^x}{(x-1)!} - e^{\lambda e^k} \lambda (e^k - 1) \Big)\\
&= \lambda \big( k e^k - e^k + 1\big).
\end{align*}
Further, it can be easily checked that the detailed balance condition \eqref{eq:dbforbd} yields that
\begin{align*}
\mathcal{E}(f_k,\log f_k)&= \sum_{x \in \N_0} a(x) \big(f_k(x+1)-f_k(x)\big)\big(\log (f_k (x+1))-\log (f_k (x))\big) \pi(x)\\
&=\frac{\lambda}{e^{\lambda e^k}} \sum_{x \in \N_0} \big( e^{k(x+1)} - e^{kx} \big)\big( k(x+1) - kx \big) \frac{\lambda^x}{x!} \\
&= \lambda k  (e^k - 1).
\end{align*}
From this, we can see that
\begin{equation*}
\frac{\mathrm{Ent}_\mu(f_k)}{\mathcal{I}(f_k)} \to 1, \, \text{ as } k \to \infty. 
\end{equation*}
Consequently, if $\Phi$  grows slower than linear at $\infty$, $EI(\Phi)$ fails for the Poisson case of a birth-death process. In particular, the Poisson case of a birth-death process  not only does not satisfy the $CD_\Upsilon(\kappa,n)$ condition (cf. Example \ref{ex:birthdeath}), but also fails on the level of the corresponding entropy-information inequality. 
\end{example}
In a somewhat similar fashion, the next result shows quite remarkable consequences of boundedness of the growth function and of the mapping $M_1$, respectively.
\begin{theorem}\label{statespacetheorem}
Let $L$ satisfy $EI(\Phi)$, then the following assertions hold true.
\begin{itemize}
\item[(i)] If $\Phi$ is bounded, then the state space $X$ is finite and the estimate 
\begin{equation}\label{eq:Entvsl1bound}
\mathrm{Ent}_\mu(f) \leq \lim\limits_{r \to \infty}\Phi(r) \Vert f \Vert_1
\end{equation}
holds true for any positive $f \in \R^X$.
\item[(ii)] If $M_{1,\sup}<\infty$ and the state space $X$ is infinite, then $\Phi$ grows linearly as $r \to \infty$.
\end{itemize}
\end{theorem}
\begin{proof}
We consider for $x \in X$ and some $\varepsilon \in (0,1)$ the function
\begin{equation*}
f_x(y) = \left\{\begin{array}{ll}
\varepsilon ,& y \neq x \\
\frac{1-\varepsilon(1-\pi(x))}{\pi(x)} ,& y=x
\end{array}\right. .
\end{equation*}
One readily verifies that $f_x \in P_*(X)$ for any $x\in X$. We have
\begin{align}\label{eq:Entfx}
\mathrm{Ent}_\mu(f_x) = \int_X f_x \log f_x\, \mathrm{d}\mu = \big(1-\varepsilon(1-\pi(x))\big) \log \frac{1-\varepsilon(1-\pi(x))}{\pi(x)} + \varepsilon \log \varepsilon \,(1-\pi(x)),
\end{align}
for any $x\in X$.  Moreover, we observe that
\begin{align*}
2 \mathcal{I}(f_x) &= \sum_{z,y \in  X} k(z,y) \big( f_x(y)-f_x(z)\big) \big( \log f_x(y) - \log f_x(z) \big) \pi(z) \\
&= \sum_{y \in X} k(x,y) \big( f_x(y)-f_x(x)\big) \big( \log f_x(y) - \log f_x(x) \big) \pi(x) \\
&\quad+\sum_{z \in X} k(z,x) \big( f_x(x)-f_x(z)\big) \big( \log f_x(x) - \log f_x(z) \big) \pi(z)\\
&= 2 \pi(x) \sum_{y \in X} k(x,y) \big( f_x(y)-f_x(x)\big) \big( \log f_x(y) - \log f_x(x) \big),
\end{align*}
where we have applied the detailed balance condition in the last step. Hence, we conclude for any $x \in X$ that
\begin{equation}\label{eq:Fisherfx}
\begin{split}
\mathcal{I}(f_x)&= \pi(x) M_1(x) \Big( \varepsilon - \frac{1-\varepsilon(1-\pi(x))}{\pi(x)}\Big) \log  \frac{\varepsilon \pi(x)}{1-\varepsilon(1-\pi(x))}\\ 
&= M_1(x) (1-\varepsilon) \log  \frac{1-\varepsilon(1-\pi(x))}{\varepsilon \pi(x)}.
\end{split}
\end{equation}
After this preliminary work we now show the first assertion. The estimate \eqref{eq:Entvsl1bound} follows from the definition of $EI(\Phi)$ for any $f \in \mathcal{P}_*(X)$ with $\mathrm{Ent}_\mu(f)<\infty$ and $\mathcal{I}(f)<\infty$. We then extend \eqref{eq:Entvsl1bound} to the more general case of positive $f\in \ell^1(\mu)$ with $\mathrm{Ent}_\mu(f)<\infty$ and $\mathcal{I}(f)<\infty$  by applying \eqref{eq:Entvsl1bound} to $\frac{f}{\Vert f \Vert_1}$. It remains to show that $X$ is finite. Assuming for contradiction that $X$ is infinite, we find a sequence $(x_m)_{m \in \N} \subset X$ such that $\pi(x_m)\to 0$ as $m \to \infty$, since $\mu$ is assumed to be a probability measure. We infer from \eqref{eq:Entfx} that $\mathrm{Ent}_\mu(f_{x_m})\to \infty$ as $m \to \infty$, which contradicts what has been shown before.

Let us now turn to the second assertion. Choosing a sequence $(x_m)_{m \in \N} \subset X$ as above and assuming w.l.o.g. that $\pi(x_m)<1$ for any $m \in \N$, we read from \eqref{eq:Entfx} that
\begin{align*}
\mathrm{Ent}_\mu(f_{x_m}) \geq  \log \frac{1}{\pi(x_m)} + \log(1-\varepsilon) + \varepsilon \log \varepsilon
\end{align*}
and from \eqref{eq:Fisherfx} that
\begin{align*}
\mathcal{I}(f_{x_m}) &\leq M_{1,\sup} (1-\varepsilon) \big( \log  \frac{1}{\pi(x_m)} - \log  \frac{\varepsilon}{1-\varepsilon(1-\pi(x_m))}\big) \\ 
&\leq M_{1,\sup} (1-\varepsilon) \big( \log \frac{1}{\pi(x_m)} -\log \varepsilon \big).
\end{align*}
Thus, we have
\begin{equation*}
\frac{\mathrm{Ent}_\mu(f_{x_m})}{\mathcal{I}(f_{x_m})} \geq \frac{ \log \frac{1}{\pi(x_m)} + \log(1-\varepsilon) + \varepsilon \log \varepsilon}{ M_{1,\sup} (1-\varepsilon) \big( \log \frac{1}{\pi(x_m)} -\log \varepsilon \big)}.
\end{equation*}
The right hand side of the latter estimate converges to $\frac{1}{(1-\varepsilon)M_{1,\sup}}$ as $m\to \infty$. This yields that there exists a constant $C>0$ and some $M \in \N$ such that
\begin{equation*}
\frac{\mathrm{Ent}_\mu(f_{x_m})}{\mathcal{I}(f_{x_m})} \geq C >0
\end{equation*}
for any $m \geq M$. Consequently, it follows from $EI(\Phi)$ that 
\begin{equation*}
C\, \mathcal{I}(f_{x_m}) \leq \Phi(\mathcal{I}(f_{x_m}))
\end{equation*}
for all $m \geq M$. Since $\mathcal{I}(f_{x_m})\to \infty$ as $m \to \infty$ and $\Phi$ is concave, we conclude that $\Phi(r)$ must grow linearly as $r \to \infty$.
\end{proof}

While the corresponding growth function in the case of the $CD_\Upsilon(\kappa,n)$ condition (with $\kappa>0$ and $n<\infty$) grows logarithmically at $\infty$, we have seen in Proposition \ref{prop:EIforpowertype} that for power type $CD$-functions of higher order the respective growth function is bounded. Combining Proposition \ref{prop:EIforpowertype} with Theorem \ref{statespacetheorem} leads to the following interesting observation.
\begin{corollary}\label{finitespaceforhigherpowers}
Let $M_1 \in \ell^2(\mu)$, $M_2 \in \ell^1(\mu)$ and the Markov chain generated by $L$ be positive recurrent. Further, let $L$ satisfy $CD_\Upsilon(\kappa,F)$ with  $\kappa>0$ and $F(r)=\frac{r^{1+\delta}}{n}$, $r \geq 0$, for some $\delta\geq 1$ and $n \in (0,\infty)$. Then the state space $X$ is finite if and only if either $\delta>1$ or $\delta=1$ and $M_{1,\sup}<\infty$.
\end{corollary}

\section{Ultracontractive Bounds under Entropy-Information inequalities}\label{sec:ultracontractive}
In the classical diffusive setting,  entropy-energy inequalities imply under the condition that $r \mapsto \frac{\Phi'(r)}{r}$ is integrable at $\infty$ ultracontractivity of the semigroup, cf. \cite{BGL}. But as the entropy-information inequality compares to entropy-energy inequalities like the modified logarithmic Sobolev inequality to logarithmic Sobolev inequalities, it is natural to expect that ultracontractive bounds come in the form of the hypercontractivity bounds from \cite{BoTe} for the modified logarithmic Sobolev inequality, i.e. not with respect to the respective norm of the semigroup, but of $e^{P_t f}$ instead. We recall the following auxiliary result, whose proof is contained in the proof of \cite[Theorem 7.1]{BoTe}, where the authors have considered an even more general setting.  Note in fact that the assumptions (1)--(4) of \cite[Section 7]{BoTe} are satisfied provided that $f \in \ell^\infty(X)$. Moreover, $f\in \ell^\infty(X)$ implies that  we have $e^{\eta(t)P_t f} \in \ell^{\infty,+}(X)$ for any fixed $t >0$ (with $\eta(t)\in \R $), which yields that $\mathcal{I}(e^{\eta(t)P_t f}) <\infty$ if we assume in addition that $M_1 \in \ell^1(\mu)$.
\begin{lemma}\label{lem:BoTeLemma}
Let $M_1 \in \ell^1(\mu)$, $f \in \ell^\infty(X)$, $t>0$ and  $q:(0,C_0)\to (0,\infty)$ be some differentiable mapping, where $C_0 \in (0,\infty]$. Then we have
\begin{equation*}
q(t) \Vert e^{P_t f}\Vert^{q(t)-1}_{q(t)}\frac{d}{dt}\Vert e^{P_t f}\Vert_{q(t)} = \frac{q'(t)}{q(t)}\mathrm{Ent}_\mu(e^{q(t)P_t f})-\mathcal{I}(e^{q(t)P_t f}).
\end{equation*} 
\end{lemma}
\begin{proof}
We briefly repeat the calculation of \cite{BoTe} for the reader's convenience and refer for more details to \cite[Section 7]{BoTe}. Note that the assumption of $M_1 \in \ell^1(\mu)$ in fact justifies to interchange integration and differentation in the lines below. We have
\begin{align*}
\frac{d}{dt}\Big(\int_X e^{q(t)P_t f}&\mathrm{d}\mu\Big)^{\frac{1}{q(t)}}\\ &= \Vert e^{P_t f} \Vert_{q(t)} \bigg( \frac{\int_X e^{q(t)P_t f} \big(q'(t)P_t f + q(t) L P_t f \big) \mathrm{d}\mu}{q(t)\int_X e^{q(t)P_t f}\mathrm{d}\mu} - \frac{q'(t)\log \int_X e^{q(t)P_t f}\mathrm{d}\mu}{q(t)^2}\bigg)\\
&= \frac{\Vert e^{P_t f} \Vert_{q(t)}^{1-q(t)}}{q(t)}\Big(\ \frac{q'(t)}{q(t)}\mathrm{Ent}_\mu(e^{q(t)P_t f})-\mathcal{I}(e^{q(t)P_t f}) \Big).
\end{align*}
\end{proof}
\begin{theorem}\label{ultracontractivityresult}
Let $L$ satisfy $EI(\Phi)$ and $M_1 \in \ell^1(\mu)$. Then for every $1 \leq p \leq q \leq \infty$, every $f \in \ell^\infty(X)$ and every $\varrho>0$
\begin{equation*}
\Vert e^{P_{t(\varrho)}f} \Vert_{q} \leq \Vert e^f \Vert_p e^{m(\varrho)}
\end{equation*}
holds true, where
\begin{equation}\label{eq:trhomrho}
t(\varrho)= \int_p^q \frac{\Phi'(\varrho r)}{r}\mathrm{d}r, \,\,
m(\varrho)= \frac{\Phi(\varrho p)}{p}-\frac{\Phi(\varrho q)}{q}.
\end{equation}
Here the case of $q=\infty$ has to be understood in the limit  $q\to \infty$ in both formulas in \eqref{eq:trhomrho}  and can be reached only if $r\mapsto \frac{\Phi'(r)}{r}$ is integrable at $\infty$.
\end{theorem}
\begin{proof}
We define $\Lambda(t)=\Vert e^{P_t f}\Vert_{q(t)}$ for a strictly increasing and differentiable  $q:(0,C_0)\to (0,\infty)$, which, together with $C_0$, will be specified below.  By Lemma \ref {lem:BoTeLemma} we have
\begin{equation*}
q(t) \Lambda(t)^{q(t)-1} \Lambda'(t) = \frac{q'(t)}{q(t)}\mathrm{Ent}_\mu(e^{q(t)P_t f}) - \mathcal{I}(e^{q(t)P_t f}).
\end{equation*}
Applying $EI(\Phi)$ in the form of \eqref{eq:EIPhinonnormalized} to $e^{q(t)P_t f}$  yields for any $r>0$
\begin{equation*}
q(t) \Lambda(t)^{q(t)-1} \Lambda'(t)\leq \mathcal{I}(e^{q(t)P_t f})\big( \frac{q'(t)}{q(t)}\Phi'(r)-1\big) + \frac{q'(t)}{q(t)}\Theta(r) \Lambda(t)^{q(t)}.
\end{equation*}
For given $r=r(q)$ (which will be made precise below) we choose $q(t)$ such that the differential equation $q'\Phi'(r(q))=q$ is satisfied, which in fact can be done by separation of variables. Indeed, let $T:(p,\infty)\to (0,C_0)$ be defined as $T(s)= \int_p^s \frac{\Phi'(r(q))}{q}\mathrm{d}q$, where $C_0=\int_p^\infty \frac{\Phi'(r(q))}{q}\mathrm{d}q$ (the value $C_0=\infty$ is allowed). Then $T$ is bijective and  $q:(0,C_0)\to (p,\infty)$, given by $q(t)=T^{-1}(t)$, solves the ODE mentioned above on $(0,C_0)$. In particular,  $q(0)=p$ extends $q$ continuously onto $[0,C_0)$.  We conclude that 
\begin{equation*}
\Lambda'(t) \leq \frac{q'(t)}{q(t)^2}\Theta(r(q(t))) \Lambda(t)
\end{equation*}
holds for any $t\in(0,C_0)$, which is equivalent to the differential inequality
\begin{equation}\label{eq:diffineqofrultra}
(\log \Lambda)' \leq \frac{q'}{q^2}\Theta(r(q)).
\end{equation}
Integrating \eqref{eq:diffineqofrultra} yields
\begin{align*}
\log (\Lambda (t)) &\leq \log(\Lambda(0)) + \int_0^t \frac{q'(s) \Theta(r(q(s)))}{q(s)^2}\mathrm{d}s \\
&= \log(\Lambda(0)) + \int_{q(0)}^{q(t)} \frac{\Theta(r(q))}{q^2}\mathrm{d}q.
\end{align*}
Note that we can write
\begin{align*}
t= \int_0^t \frac{q'(s)\Phi'(r(q(s)))}{q(s)}\mathrm{d}s = \int_{q(0)}^{q(t)} \frac{\Phi'(r(q))}{q}\mathrm{d}q.
\end{align*}
Choosing $r(q)= \varrho q$ establishes the formula for $t(\varrho)$. Moreover, recalling that $\Theta(s)=\Phi(s)-\Phi'(s)s$, $s \in (0,\infty)$, we deduce from a simple application of integration by parts that
\begin{align*}
\int_p^q \frac{\Theta(\varrho s)}{s^2}\mathrm{d}s = \int_p^q \frac{\Phi(\varrho s)}{s^2}\mathrm{d}s - \varrho t(\varrho) = \frac{\Phi(\varrho p)}{p} - \frac{\Phi(\varrho q)}{q}, 
\end{align*}
which yields the claim.
\end{proof}
We see from Theorem \ref{ultracontractivityresult} that $q=\infty$ can be reached provided that $\frac{\Phi'(r)}{r}$ is integrable at $\infty$. In particular, in case of the modified logarithmic Sobolev inequality, Theorem \ref{ultracontractivityresult} does not lead to ultracontractive bounds, which is consistent to the role of the logarithmic Sobolev inequality in the diffusive setting. Otherwise, for a growth function $\Phi$ that behaves as $r^\alpha$ with $0<\alpha<1$ at $\infty$ we have that $\frac{\Phi'(r)}{r}$ is integrable  at $\infty$. In this sense, the modified logarithmic Sobolev inequality constitutes an extreme case.

The growth function resulting from $CD_\Upsilon(\kappa,F)$, with $F$ being a power-type $CD$-function from Proposition \ref{prop:EIforpowertype}, satisfies the integrability condition that we have mentioned throughout the previous lines. This fact can be seen from the identity \eqref{eq:derivativepowertypegrowth}.  We close this section with an application of Theorem \ref{ultracontractivityresult} in this particular context.

\begin{corollary}
Let $M_1 \in \ell^2(\mu)$, $M_2 \in \ell^1(\mu)$ and the Markov chain generated by $L$ be positive recurrent. Further, let $L$ satisfy $CD_\Upsilon(\kappa,F)$ with $\kappa>0$ and $F(r)=\frac{r^{1+\delta}}{n}$, $r\geq 0$, for some $\delta\geq 1$ and $n \in (0,\infty)$. Then we have for any $t>0$ that
\begin{equation}\label{eq:ultraconpowertype}
\Vert e^{P_t f} \Vert_\infty \leq e^{\Phi\big(\sqrt[\delta]{\frac{n}{2\delta t}}\big)} \Vert e^f \Vert_1
\end{equation}
holds for any $f \in \ell^\infty(X)$, where $\Phi$ denotes the growth function given by \eqref{eq:powertypegrowthfunction}.
In particular, in case of the $CD_\Upsilon(\kappa,n)$ condition, \eqref{eq:ultraconpowertype} reads as
\begin{equation}\label{eq:ultraconCDsquare}
\Vert e^{P_t f} \Vert_\infty \leq \Big( 1+ \frac{1}{2\kappa t}\Big)^\frac{n}{2} \Vert e^f \Vert_1 .
\end{equation} 
\end{corollary}
\begin{proof}
Due to Proposition \ref{prop:EIforpowertype}, $L$ satisfies $EI(\Phi)$ with growth function given by \eqref{eq:powertypegrowthfunction}. We choose $p=1$ and $q=\infty$ (in the limit sense) in \eqref{eq:trhomrho}, recall the formula \eqref{eq:derivativepowertypegrowth} for the derivative of the growth function and observe for $\varrho>0$
\begin{equation*}
t(\varrho)= \frac{n}{2}\int_1^\infty \frac{1}{r(\kappa n +(\varrho r)^\delta)}\mathrm{d}r \leq \frac{n}{2\varrho^\delta}\int_1^\infty \frac{1}{r^{1+\delta}}\mathrm{d}r = \frac{n}{2\delta \varrho^\delta}.
\end{equation*}
From this we infer by monotonicity of the growth function that
\begin{equation*}
m(\varrho)= \Phi(\varrho) \leq \Phi\Big(\sqrt[\delta]{\frac{n}{2\delta t(\varrho)}}\Big).
\end{equation*}
Consequently, by Theorem \ref{ultracontractivityresult} we get that
\begin{equation*}
\Vert e^{P_{t(\varrho)} f} \Vert_\infty \leq e^{\Phi\big(\sqrt[\delta]{\frac{n}{2\delta t(\varrho)}}\big)} \Vert e^f \Vert_1
\end{equation*}
holds for any $\varrho>0$. But as the mapping $\varrho \mapsto t(\varrho)$, $\varrho>0$, is bijective onto $(0,\infty)$ since it is decreasing with $t(\varrho)\to 0$ as $\varrho \to \infty$ and $t(\varrho)\to \infty$ as $\varrho \to 0$, \eqref{eq:ultraconpowertype} follows. The special case of \eqref{eq:ultraconCDsquare} now can be established by the explicit formula for the growth function in the case of $\delta=1$, see Corollary \ref{logEIunderCD(kappa,n)}.
\end{proof}
\section{Exponential Integrability of Lipschitz functions and Diameter bounds}\label{sec:diameter}
Exponential integrability of Lipschitz functions and diameter bounds (see the definitions below) are both important properties to investigate in the classical theory of \cite{BGL}.  In order to reach finite diameter bounds in the diffusive setting,  Poincar\'e inequalities resp. logarithmic Sobolev inequalities are not sufficient. Instead, Sobolev inequalities resp. certain entropy-energy inequalities ensure the validity of a finite diameter. Speaking on the level of CD-inequalities this means that in the diffusive setting positive curvature and finite dimension suffices to deduce finite diameter bounds, while positive curvature alone does not. In this section we will be able to show that the $CD_\Upsilon$ condition behaves consistently in the discrete setting of Markov chains. 

Now, we recall the definitions of Lipschitz functions and the diameter.
\begin{defi}\label{def:Lipschitz}
A function $f\in \R^X$ is called Lipschitz function if $\Gamma(f)(x)$ exists at any $x \in X$ (in the sense that the sum in \eqref{eq:PsiH} with $H(r)=\frac{r^2}{2}$ is finite) and $\Vert f \Vert_{\mathrm{Lip}}:= \sqrt{\Vert \Gamma(f) \Vert_\infty} < \infty$. Moreover, we say that $f$ is $C$-Lipschitz, where $C>0$, when $\Vert f \Vert_{\mathrm{Lip}}\leq C$ holds true.
\end{defi}
\begin{defi}\label{def:diameter}
Considering the mapping
$\varrho: X \times X \to [0,\infty)$, given by
\begin{equation*}
\varrho (x,y) = \sup \{f(y)-f(x) : \Vert f \Vert_{Lip} \leq 1 \},
\end{equation*}
we define the diameter with respect to $L$ as $\mathrm{diam}_{\varrho}= \sup\limits_{x,y \in X}\varrho(x,y)$.
\end{defi}
These definitions have been used in the diffusive situation of \cite{BGL}, but also in the discrete setting in \cite{LMP}, where $\mathrm{diam}_\varrho$ has been called the resistance diameter. Definition \ref{def:diameter} is further closely related to the diameter with respect to the combinatorical graph distance on the underlying graph  to $L$. In fact, in \cite{LMP} it has been shown that the estimate
\begin{equation}\label{eq:diametercomparison}
\mathrm{dist}(x,y)\leq \sqrt{\frac{M_{1,\sup}}{2}}\varrho(x,y)
\end{equation}
holds true on locally finite graphs, provided that $M_{1,\sup}<\infty$, where $\mathrm{dist}:X\times X\to[0,\infty)$ denotes the combinatorical graph distance. We emphasize that the bound \eqref{eq:diametercomparison} extends to the case of locally infinite graphs with $M_{1,\sup}<\infty$ as the proof of \cite[Lemma 1.4]{LMP} holds verbatim. 

\begin{lemma}\label{lem:FisherLipschitz}
Let $f \in \R^X$ be a bounded $C$-Lipschitz function. Then we have that 
\begin{equation*}
\mathcal{I}(e^{sf}) \leq C^2\, s^2 \int_X e^{sf}\mathrm{d}\mu.
\end{equation*}
holds true for any $s\in \R$.
\end{lemma}
\begin{proof}
Clearly, there is nothing to show in the case of $s=0$.
Note that the detailed balance condition implies that 
\begin{equation*}
k(x,y) s\big(f(y)-f(x)\big)\big(h(sf(y))-h(sf(x))\big)\pi(x)= k(y,x) s\big(f(x)-f(y)\big)\big(h(sf(x))-h(sf(y))\big)\pi(y)
\end{equation*}
holds for any $x,y\in X$ and $s\neq 0$, where $h:\R \to \R$ is some arbitrary function. Defining for $x\in X$ and $s\neq 0$ the set $A_{x,s}:= \{y \in X: s(f(y)-f(x))<0 \}$ , we then infer
\begin{align*}
C^2 s^2 \int_X e^{sf} \mathrm{d}\mu &\geq s^2 \int_X e^{sf}\Gamma(f) \mathrm{d}\mu \\
&= s^2 \sum_{x \in X} e^{sf(x)} \sum_{y \in A_{x,s}} k(x,y) \big(f(y)-f(x)\big)^2\pi(x)\\
&\geq \sum_{x \in X}e^{sf(x)} \sum_{y \in A_{x,s}}k(x,y) |s(f(y)-f(x))|\,|e^{s(f(y)-f(x))}-1|\,\pi(x)\\
&= \sum_{x \in X} e^{s f(x)} \sum_{y \in A_{x,s}} k(x,y) s\big(f(y)-f(x)\big)\big(e^{s(f(y)-f(x))}-1\big) \pi(x) \\
&= \sum_{x \in X} \sum_{y \in A_{x,s}}k(x,y)s\big(f(y)-f(x)\big)\big(e^{sf(y)}-e^{sf(x)}\big)\pi(x) \\
&= \mathcal{I}(e^{sf}).
\end{align*}
Here, we have applied the inequality $|\tau|\geq|e^\tau - 1|$, which is valid if $\tau\leq 0$.
\end{proof}
We now state the main result of this section.
\begin{theorem}\label{expintegrabilityforLipschitz}
Let $L$ satisfy $EI(\Phi)$ and let $f \in \ell^1(\mu)$ be $1$-Lipschitz. Further, we assume that $s\mapsto \frac{\Phi(s^2)}{s^2}$ is integrable at $0$. Then we have
\begin{equation}\label{eq:expintegrabilityLip}
\int_X e^{tf} \mathrm{d}\mu \leq \exp \Big( t \Big( \int_0^t \frac{\Phi(s^2)}{s^2}\mathrm{d}s + \int_X f \mathrm{d}\mu \Big) \Big),\,\, t \neq 0.
\end{equation}
If in addition $s\mapsto \frac{\Phi(s^2)}{s^2}$ is also integrable at $\infty$, then any Lipschitz function is bounded and it holds
\begin{equation}\label{eq:meanvalueestimate}
\Vert f - \int_X f \mathrm{d}\mu \Vert_\infty \leq \int_0^\infty \frac{\Phi(s^2)}{s^2}\mathrm{d}s
\end{equation}
 for any $1$-Lipschitz function $f$.\\
In particular, the diameter bound
 \begin{equation}\label{eq:Diameterbound}
\mathrm{diam_\varrho} \leq 2  \int_0^\infty \frac{\Phi(s^2)}{s^2}\mathrm{d}s
\end{equation}
is valid.
\end{theorem}
\begin{proof}
First, let $f \in \R^X$ be $1$-Lipschitz and bounded. In particular, by Lemma \ref{lem:FisherLipschitz}  we have that $\mathcal{I}(e^{sf})<\infty$ for any $s\in \R$. We set $Z(s)= \int_X e^{sf}\mathrm{d\mu}$ and observe that $Z'(s)= \int_X f e^{sf}\mathrm{d}\mu$. Further, we define  $\Lambda(s)=\frac{\log ( Z(s)) }{s}$, $s \neq 0$. Note that we can extend $\Lambda$ to a continuous function on $\R$ by means of L'Hospital's rule by $\Lambda(0)= \int_X f \mathrm{d}\mu$. For some fixed $s\neq 0$ we obtain 
\begin{align*}
\Lambda'(s) = \frac{s Z'(s)-Z(s) \log(Z(s))}{s^2 Z(s)} = \frac{\mathrm{Ent}_\mu(e^{sf})}{s^2 Z(s)}.
\end{align*}
Now, we apply $EI(\Phi)$ (in the form of \eqref{eq:EIPhinonnormalized}) to $e^{sf}$ and observe that
\begin{equation}\label{eq:LambdaderivativeforLipschitz}
\Lambda'(s) \leq \frac{\Phi'(r)\mathcal{I}(e^{sf})+\Theta(r)Z(s)}{s^2 Z(s)}
\end{equation}
holds for any $r \in (0,\infty)$. Due to $f$ being bounded and $1$-Lipschitz, Lemma \ref{lem:FisherLipschitz} translates to
\begin{equation*}
\mathcal{I}(e^{sf}) \leq s^2 Z(s).
\end{equation*}
With this at hand, we deduce from \eqref{eq:LambdaderivativeforLipschitz} that
\begin{equation*}
\Lambda'(s) \leq \frac{\Phi'(r)s^2 + \Theta(r)}{s^2} = \frac{\Phi(r)+\Phi'(r)(s^2 - r)}{s^2}
\end{equation*}
holds for any $r \in (0,\infty)$. Specifying $r=s^2$, we end up with
\begin{equation*}
\Lambda'(s) \leq \frac{\Phi(s^2)}{s^2}.
\end{equation*}
Integrating this yields
\begin{align*}
\Lambda(t) \leq \Lambda(0) + \int_0^t \frac{\Phi(s^2)}{s^2}\mathrm{d}s,
\end{align*}
when $t>0$ and with the reverse inequality in case that $t<0$. In both situations, \eqref{eq:expintegrabilityLip} follows from the definition of $\Lambda$ in the case that $f$ is bounded.

Regarding the general case, let $f\in \ell^1(\mu)$ be $1$-Lipschitz and consider  $f_N \in \ell^\infty(X)$, $N \in \N$, given by $f_N(x)=f(x)$ if $|f(x)|\leq N$, $f_N(x)=N$ if $f(x)>N$ and $f_N(x)=-N$ if $f(x)<-N$. Clearly, for any $x,y \in X$, we have $|f_N(x)-f_N(y)| \leq|f(x)-f(y)|$ and thus $f_N$ is $1$-Lipschitz for any $N \in \N$. As $f \in \ell^1(\mu)$, $\int_X f_N \mathrm{d}\mu$ converges to $\int_X f \mathrm{d}\mu$ as $N \to \infty$ by the dominated convergence theorem. Furthermore, Fatou's lemma implies
\begin{align*}
\int_X e^{sf}\mathrm{d}\mu \leq \liminf_{N \to \infty} \int_X e^{s f_N} \mathrm{d}\mu &\leq \liminf_{N \to \infty} \exp \Big( t \Big( \int_0^t \frac{\Phi(s^2)}{s^2}\mathrm{d}s + \int_X f_N \mathrm{d}\mu \Big) \Big) \\
&= \exp \Big( t \Big( \int_0^t \frac{\Phi(s^2)}{s^2}\mathrm{d}s + \int_X f \mathrm{d}\mu \Big) \Big).
\end{align*}

As a first step to establish the second claim, we note that if \eqref{eq:meanvalueestimate} is valid for bounded $1$-Lipschitz functions then any $1$-Lipschitz function is an element of $\ell^1(\mu)$. Indeed, let $f$ be a general $1$-Lipschitz function and denote by $\big( f_N\big)_{N \in \N}\subset \ell^\infty(X)$ the approximating sequence of $1$-Lipschitz functions as defined above. Then, choosing $x \in X$ arbitrary, we observe, since clearly $|f|$ is also $1$-Lipschitz, that
\begin{align*}
\big| |f (x)| - \int_X |f_N| \mathrm{d}\mu\, \big| \leq \big\Vert |f_N| - \int_X |f_N| \mathrm{d}\mu \,\big\Vert_\infty \leq \int_0^\infty
\frac{\Phi(s^2)}{s^2}\mathrm{d}s,
\end{align*}
where $N\geq N_0$ with $\N_0\in \N$ such that $|f(x)|\leq N_0$. Sending $N\to \infty$ in the latter estimation, we deduce that $f \in \ell^1(\mu)$.

We will now show that \eqref{eq:meanvalueestimate} holds for any $1$-Lipschitz function that is an element of $\ell^1(\mu)$ (and hence for any $1$-Lipschitz function).
We set $C= \int_0^\infty \frac{\Phi(s^2)}{s^2}\mathrm{d}s$ and apply \eqref{eq:expintegrabilityLip} to $f - \int_X f \mathrm{d}\mu$, assuming that $f\in \ell^1(\mu)$ is $1$-Lipschitz. We obtain for any $t>0$ that
\begin{equation}\label{eq:secondclaimDiameterbound}
\int_X \exp\Big(t \big( f - \int_X f \mathrm{d}\mu\big)\Big)\mathrm{d}\mu \leq \exp \big( t C\big)   .
\end{equation}
Now, we assume for contradiction that we can find some $x \in X$ and some $\varepsilon>0$ such that 
\begin{equation*}
f(x)-\int_X f \mathrm{d}\mu > C+\varepsilon.
\end{equation*} 
Then we have for $t>0$
\begin{align*}
\int_X \exp\Big(t \big( f - \int_X f \mathrm{d}\mu\big)\Big)\mathrm{d}\mu > e^{t(C+\varepsilon)}\pi(x), 
\end{align*}
which contradicts \eqref{eq:secondclaimDiameterbound} in the asymptotic behavior of $t \to \infty.$ Simultaneously, by considering the asymptotic behavior as $t \to -\infty$, one obtains that $f(x) - \int_X f \mathrm{d}\mu \geq - C$ holds for all $x \in X$. This establishes \eqref{eq:meanvalueestimate} for any $1$-Lipschitz function. From that we conclude
\begin{equation}\label{eq:lipschitzfxfybound}
f(y)-f(x) \leq 2 \Vert f - \int_X f \mathrm{d}\mu\Vert_\infty \leq 2 \int\limits_0^\infty \frac{\Phi(s^2)}{s^2}\mathrm{d}s.
\end{equation}
for any $x,y \in X$ and $1$-Lipschitz function $f$. But \eqref{eq:lipschitzfxfybound} implies that $f$ must be bounded. Hence we deduce by scaling that any Lipschitz function is bounded. Furthermore, by \eqref{eq:lipschitzfxfybound} and the definition of $\varrho$, we deduce \eqref{eq:Diameterbound}.
\end{proof}
Theorem \ref{expintegrabilityforLipschitz} yields finite bounds on $\mathrm{diam}_\varrho$ if $L$ satisfies $CD_\Upsilon(\kappa,F)$ with $\kappa>0$ and a power-type $CD$-function by means of Proposition \ref{prop:EIforpowertype}(ii). In the special case of the quadratic $CD$-function, we get the following bound.
\begin{corollary}\label{Diameterbound}
If $M_1 \in \ell^2(\mu)$, $M_2 \in \ell^1(\mu)$, the Markov chain generated by $L$ is positive recurrent and $L$ satisfies $CD_\Upsilon(\kappa,n)$ with $\kappa>0$ and $n<\infty$, then the diameter bound 
\begin{equation}\label{diameterboundunderCD}
\mathrm{diam}_\varrho\leq \pi\sqrt{\frac{n}{\kappa}}
\end{equation}
holds true.
\end{corollary}
\begin{proof} 
By elementary methods one calculates the integral
\begin{equation*}
\frac{n}{2}\int_0^\infty \frac{\log\big(1+\frac{s^2}{\kappa n}\big)}{s^2}\mathrm{d}s = \frac{\pi}{2}\sqrt{\frac{n}{\kappa}}.
\end{equation*}
The claim follows by combining Corollary \ref{logEIunderCD(kappa,n)} with Theorem \ref{expintegrabilityforLipschitz}.
\end{proof}
\begin{remark}\label{rem:LMPcomparision}
By \eqref{diameterboundunderCD} we recover (by different methods) exactly the same diameter bound as in \cite{LMP}, where there it is assumed on the one hand only $CD(\kappa,n)$ but on the other hand that the underlying graph to $L$ is locally finite and satisfies the completeness assumption and non-degeneracy of the vertex measure. Note that by \cite[Theorem 2.2]{shuLiu} the latter boils down to the case of finite graphs since $\kappa>0$. Hence, although the curvature-dimension condition  of Corollary \ref{Diameterbound} is more restrictive, the setting where it applies can be expected to be more general compared to the one of \cite{LMP}.
\end{remark}
With the following example we aim to emphasize that the property that Lipschitz functions are bounded is quite strong in the sense that it fails for a large class of examples that all satisfy corresponding modified logarithmic Sobolev inequalities. In particular, it turns out that $CD_\Upsilon(\kappa,\infty)$, with $\kappa>0$, is not sufficient for deducing a finite diameter bound.
\begin{example}
We consider a birth-death process on $\N_0$ as in Example \ref{ex:birthdeath} and employ the notation that has been used therein. In particular, we assume that the rate functions $a$ and $b$ are monotone as in Example \ref{ex:birthdeath} and that condition \eqref{eq:Caputoassumption} holds for some $\kappa>0$. We set $f(0)=0$ and define the sequence of partial sums
\begin{equation}\label{eq:bdprocesspartialsums}
f(n)=\sum_{k=1}^n \frac{1}{\sqrt{b(k)}}, \, n \in \N.
\end{equation}
We claim that Lipschitz functions are bounded if and only if the partial sums given by \eqref{eq:bdprocesspartialsums} converge.

First, we assume that the partial sums given by \eqref{eq:bdprocesspartialsums} diverge as $n \to \infty$. Then, we have
\begin{equation*}
2\Gamma (f) (n) = a(n) (f(n+1)-f(n))^2 + b(n)(f(n-1)-f(n))^2 = \frac{a(n)}{b(n+1)} + 1
\end{equation*}
for any $n \in \N$. From the monotonicity assumption on the rates we infer that $\Gamma(f)$ is bounded, or in other words that $f$ is a Lipschitz function. But apparently, $f$ is unbounded. This yields that the corresponding generator does not satisfy an entropy-information inequality with growth function $\Phi$ such that $\int_0^\infty \frac{\Phi(s^2)}{s^2}\mathrm{d}s<\infty$. On the other hand, it is known by \cite{CDP} that $L$ satisfies a corresponding modified logarithmic Sobolev inequality. Moreover, we emphasize that among those birth-death processes of the present example are also processes that even satisfy the condition $CD_\Upsilon(\kappa,\infty)$, see \cite{WZ}. This shows that the condition $CD_\Upsilon(\kappa,\infty)$ with $\kappa>0$ is in general  not sufficient to obtain a finite diameter. The latter finding is consistent to the Bakry-\'Emery condition in the diffusive setting.

Now, let us assume conversely that the partial sums given by \eqref{eq:bdprocesspartialsums} converge and let $g$ be $C$-Lipschitz for some $C>0$. In particular, 
\begin{equation*}
2 b(n) \big(g(n-1)-g(n)\big)^2 \leq C^2
\end{equation*}
holds for any $n \in \N$. Consequently, we have $|g(n-1)-g(n)|\leq \frac{C}{\sqrt{2 b(n)}}$, $n \in \N$, and by the triangle inequality we deduce
\begin{equation*}
|g(N)-g(0)|\leq \frac{C}{\sqrt{2}} f(N)\leq \frac{C}{\sqrt{2}} \sum_{n \in \N}\frac{1}{\sqrt{b(n)}}
\end{equation*}
for any $N\in \N$, which yields that $g$ is bounded.
\end{example}

\section{Modified Nash inequalities}\label{sec:Nash}
In the diffusive setting, logarithmic entropy-energy inequalities are known to be equivalent to Nash inequalities. Regarding the discrete setting of Markov chains we refer to \cite{DiSC}, and also \cite{SC}, for an extensive account on Nash inequalities. Clearly,  it can not be expected that logarithmic entropy-information inequalities  are linked to the classical Nash inequality in the discrete setting as they are in the diffusive setting by the lack of chain rule. We refer to the natural analogue as the modified Nash inequality, which can be induced from corresponding logarithmic entropy-information inequality as will be shown subsequently. We say that a function $f\in \R^X$ is non-vanishing if $f(x)\neq 0$ for any $x \in X$.
\begin{theorem}\label{modNash}
If $L$ satisfies $EI(\Phi)$ with $\Phi(r)= \alpha \log \big( A + \frac{r}{\beta}\big) $ and $\alpha,\beta>0$, $A\geq 1$, then the following modified Nash inequality
\begin{equation}\label{eq:Nash}
\Vert f \Vert_2^{2\alpha+2} \leq \Big( A\Vert f \Vert_2^2 + \frac{\mathcal{I}(f^2)}{\beta}\Big)^{\alpha} \Vert f \Vert_1^2
\end{equation}
holds for any non-vanishing $f \in \ell^2(\mu)$.
\end{theorem}
\begin{proof}
Clearly, we can assume that $\mathcal{I}(f^2)<\infty$. Further, it suffices to prove \eqref{eq:Nash} for bounded non-vanishing functions by a standard truncation argument. Indeed, let $(f_N)_{N \in \N}$ denote the sequence of bounded functions that has been considered in the proof of Theorem \ref{expintegrabilityforLipschitz}. Then it follows readily by means of the monotone convergence theorem that $\Vert f_N \Vert_2 \to \Vert f \Vert_2$, $\Vert f_N \Vert_1 \to \Vert f \Vert_1$ and $\mathcal{I}(f_N^2) \to \mathcal{I}(f^2)$ as $N\to \infty$.

It is a well known consequence of H\"older's inequality that the mapping $r\mapsto \Vert f \Vert_{\frac{1}{r}}$, $r \in (0,1]$ is log-convex.
Then, for  a given non-vanishing $f \in \ell^\infty(X)$ with $\Vert f \Vert_2 = 1$,  we consider  the convex mapping $\Lambda(r)= \log \Vert f \Vert_{\frac{1}{r}}$, $r \in (0,1]$, which is well defined since $\mu$ is a probability measure. One readily verifies by a similar calculation as in the proof of Lemma \ref{lem:BoTeLemma} that 
\begin{equation*}
\Lambda'(r)=- \frac{\mathrm{Ent}_\mu(|f|^{\frac{1}{r}})}{\int_X |f|^\frac{1}{r}\mathrm{d}\mu},
\end{equation*}
where we use that $|f|$ is bounded in order to interchange differentation and integration. In particular, using $\Vert f \Vert_2 = 1$, we observe that
 $\Lambda'\big(\frac{1}{2}\big)=- \mathrm{Ent}_\mu(f^2)$. By convexity, we thus have
\begin{equation*}
2 \big( \Lambda(1)-\Lambda\big(\frac{1}{2}\big)\big) \geq \Lambda'\big(\frac{1}{2}\big).
\end{equation*}
Consequently, by the entropy-information inequality $EI(\Phi)$ and the fact that $\Lambda\big(\frac{1}{2}\big)=0$, we observe
\begin{align*}
\log  \frac{1}{\Vert f \Vert_1^2}\leq \mathrm{Ent}_\mu(f^2) \leq \log \Big( A + \frac{\mathcal{I}(f^2)}{\beta} \Big)^\alpha,
\end{align*}
which implies
\begin{equation}\label{eq:Nashfornormalizedcase}
1 \leq \big(A+ \frac{\mathcal{I}(f^2)}{\beta}\big)^{\alpha} \Vert f \Vert_1^2\,.
\end{equation}
Now, for the non-normalized case we apply \eqref{eq:Nashfornormalizedcase} to $\frac{f}{\Vert f \Vert_2}$. By the scaling behavior of the Fisher information (cf. \eqref{eq:scalingFisher}), we deduce 
\begin{align*}
\Vert f \Vert_2^2 \leq \frac{1}{\Vert f \Vert_2^{2\alpha}} \Big( A\Vert f \Vert_2^2 + \frac{\mathcal{I}(f^2)}{\beta} \Big)^{\alpha} \Vert f \Vert_1^2,
\end{align*}
from which the claim follows.
\end{proof}
Combining Corollary \ref{logEIunderCD(kappa,n)} with Theorem \ref{modNash}, we observe the following result.
\begin{corollary}\label{modNashunderCD}
If $M_1 \in \ell^2(\mu)$, $M_2 \in \ell^1(\mu)$, the Markov chain generated by $L$ is positive recurrent and $L$ satisfies $CD_\Upsilon(\kappa,n)$ with $\kappa>0$ and $n<\infty$, then $L$ satisfies the following modified Nash inequality
\begin{equation*}
\Vert f \Vert_2^{n+2} \leq \Big( \Vert f \Vert_2^2 + \frac{\mathcal{I}(f^2)}{\kappa n}\Big)^{\frac{n}{2}} \Vert f \Vert_1^2
\end{equation*}
for any non-vanishing $f \in \ell^2(\mu)$.
\end{corollary}

\appendix
\section{Auxiliary Lemma}
In this section we provide an auxiliary result, which has been used to investigate the example of the two-point space in Example \ref{ex:2point}. For further properties of the functions $\nu_{c,d}:\R \to \R$, given by
\begin{equation*}
\nu_{c,d}(r)=c \Upsilon'(r)r+ \Upsilon(-r)-d \Upsilon(r), \, c,d \in \R,
\end{equation*}
which also have been used throughout Section \ref{sec:CDcondition}, we refer to the Appendix of \cite{WZ}.
\begin{lemma}\label{nulambdaconvex}
$\nu_{1+\lambda,\lambda}$ is strictly convex for any $\lambda \in (0,1]$.
\end{lemma}
\begin{proof}
We have
\begin{align*}
\nu_{1+\lambda,\lambda}''(r)&=e^r\big( (1+\lambda)r + 2 + \lambda\big) +e^{-r},\\
\nu_{1+\lambda,\lambda}'''(r)&= e^r\big( (1+\lambda)r+3+2\lambda \big) -e^{-r}.
\end{align*}
Hence, $\nu_{1+\lambda,\lambda}'''(r)=0$ if and only if
\begin{equation}\label{eq:nulambdathird}
e^{2r}\big( (1+\lambda)r+3+2\lambda\big) = 1,
\end{equation}
which yields by monotonicity that $\nu_{1+\lambda,\lambda}''$ has a unique critical point $r_*$ for which we have $r_* > - \frac{3+2\lambda}{1+\lambda}$. Then we can reformulate \eqref{eq:nulambdathird} by applying the logarithm into
\begin{equation}\label{eq:rstaridentity}
r_* = \log \big( \frac{1}{\sqrt{(1+\lambda)r_*+3+2\lambda}}\big).
\end{equation}
We infer that
\begin{equation*}
\nu_{1+\lambda,\lambda}''(r_*) = \frac{(1+\lambda)r_* + 2 + \lambda}{\sqrt{(1+\lambda)r_* +3 + 2\lambda}} + \sqrt{(1+\lambda)r_* +3 + 2\lambda} \, ,
\end{equation*}
which is positive if and only if 
\begin{equation*}
r_* > - \frac{5+3\lambda}{2(1+\lambda)}.
\end{equation*}
Defining $\bar{r}:=- \frac{5+3\lambda}{2(1+\lambda)}$, we now observe that 
\begin{align*}
e^{2\bar{r}} < 1 \leq \frac{2}{\lambda + 1} = \frac{1}{(1+\lambda)\bar{r}+3+2\lambda}.
\end{align*}
In particular, this implies that $\nu_{1+\lambda,\lambda}'''(\bar{r})< 0$ and hence we have $r_* > \bar{r}$ by monotonicity, which establishes the claim.
\end{proof}

\end{document}